\documentclass[12pt]{amsart}

\usepackage{amsfonts}
\usepackage{amsthm}
\usepackage{amstext}
\usepackage{amsmath}
\usepackage{amscd}
\usepackage{amssymb}
\usepackage[mathscr]{eucal}
\usepackage{graphicx}
\usepackage{graphics}
\usepackage{epsf}
\usepackage{array}
\usepackage[pdftex,colorlinks=true,allcolors=blue]{hyperref}

 \makeatletter
 \def\@seccntformat#1{\csname the#1\endcsname.\quad}
 \makeatother

\theoremstyle{plain}
\newtheorem{theorem}{Theorem}
\newtheorem{proposition}[theorem]{Proposition}

\newtheorem{lemma}[theorem]{Lemma}

\theoremstyle{remark}

\numberwithin{equation}{section}

\newcommand{\mL}{L\kern-0.08cm\char39}

\newcommand{\last}[1]{\epsilon_{#1}}            

\newcommand{\modm}[1]{\, (\operatorname{mod}\, #1)}
\newcommand{\Fix}{\operatorname{Fix}}
\newcommand{\Per}{\operatorname{Per}}
\newcommand{\Lip}{\operatorname{Lip}}

\newcommand{\id}{{\rm id}}
\newcommand{\End}{E}
\newcommand{\Ord}{O}
\newcommand{\Branch}{B}

\newcommand{\emptyword}{o}

\newcommand{\NNN}{\mathbb N}
\newcommand{\ZZZ}{\mathbb Z}
\newcommand{\RRR}{\mathbb R}

\newcommand{\BBb}{\mathcal{B}}
\newcommand{\DDd}{\mathcal{D}}

\newcommand{\diam}{\operatorname{diam}}
\newcommand{\orbit}{\operatorname{Orb}}
\newcommand{\closure}[1]{\overline{#1}}
\newcommand{\boundary}[1]{\partial {#1}}
\newcommand{\interior}[1]{\operatorname{int}(#1)}
\newcommand{\eps}{\varepsilon}
\newcommand{\abs}[1]{\lvert#1\rvert}

\newcommand{\mr}[1]{\href{http://www.ams.org/mathscinet-getitem?mr=#1}{MR#1}}
\newcommand{\doi}[1]{\href{http://dx.doi.org/#1}{doi:#1}}

\begin{document}

\title[Transitive dendrite map with infinite decomposition ideal]
{Transitive dendrite map with \\infinite decomposition ideal}

\author[Vladim\'\i r \v Spitalsk\'y]{Vladim\'\i r \v Spitalsk\'y}
\address{Department of Mathematics, Faculty of Natural Sciences,
          Matej Bel University, Tajovsk\'eho 40, 974 01 Bansk\'a Bystrica,
          Slovakia}
\email{vladimir.spitalsky@umb.sk}



\begin{abstract}
By a result of Blokh from 1984, every transitive map of a tree has the relative specification property,
and so it has finite decomposition ideal, positive entropy and dense periodic points.
In this paper we construct a transitive dendrite map with
infinite decomposition ideal and a unique periodic point.
Basically, the constructed map is (with respect to any non-atomic invariant measure) a measure-theoretic extension
of the dyadic adding machine. Together with an example of Hoehn and Mouron from 2013,
this shows that transitivity on dendrites is much more varied than that on trees.
\end{abstract}

\maketitle

\thispagestyle{empty}

\section{Introduction}\label{S:intro}

The fundamental result on topologically transitive graph maps is the theorem
of Blokh from 1984 \cite{Blo84}. It states that every transitive map $F$ on a connected
graph $X$
\begin{itemize}
  \item either is (conjugate to) an irrational rotation of the circle
  \item or has the relative specification property.
\end{itemize}
(Recall that a dynamical system
$(X,F)$ has the relative specification property if there is a regular periodic
decomposition $\DDd=(D_0,\dots,D_{m-1})$ of $X$ such that the restrictions
$F^m|_{D_i}$ of the $m$th iterate of $F$ have the specification property;
for the corresponding definitions, see the next section.)
This result was partially extended to compacta containing a free arc
(that is, an arc every non-end point of which is an interior point of the arc)
in \cite{DSS12}. It is proved there that if $X$ is a compact metric space with a free arc
and $F:X\to X$ is a transitive map, then exactly one of the following two statements holds:
\begin{itemize}
  \item $X$ is a disjoint union of finitely many circles permuted by $F$,
   and the corresponding iterate of $F$,
   restricted to any of the circles, is (conjugate to) an irrational rotation;
  \item $F$ is relatively mixing, non-injective, has dense periodic points and
   positive topological entropy.
\end{itemize}

On compacta without free arcs, such a dichotomy is not true. Counterexamples
exist even on curves.
For example, one can take any transitive aperiodic zero entropy map on the Cantor set $C$ and extend it,
using the results from \cite{AKLS,BS03}, to a transitive map on the cone over $C$
(the so-called Cantor fan) with a unique periodic point and zero entropy.
However, the Cantor fan is not locally connected and, as far as we know,
the first example of a locally connected curve for which Blokh's dichotomy is not true,
was given by Hoehn and Mouron \cite{HM13}.
They proved that on dendrites
(that is, on locally connected continua containing no circle)
weak mixing does not imply mixing.
Moreover, the map from \cite{HM13} has a unique periodic point, as was shown in \cite{AHNO13}.

The purpose of the present paper is to provide another example of a transitive
dendrite map which shows, together with \cite{HM13}, that transitivity on dendrites
is much more varied than that on trees or graphs.
In contrast with \cite{HM13}, our map, though transitive, is not (relatively) weakly mixing.
It has infinite decomposition ideal, that is, it admits an infinite sequence
of refining regular periodic decompositions.
These and other properties of the constructed dendrite map are summarized
in the following theorem.
(Recall that a map $f$ is almost open if it sends non-empty open sets to sets with
non-empty interior. As usual, $\BBb_X$ is the $\sigma$-algebra of Borel subsets of $X$.)

\begin{theorem}\label{T:main}
There is a dendrite $X$ and a continuous map $f:X\to X$ such that
\begin{enumerate}
  \item[(a)] $f$ is transitive with all transitive points being end points of $X$;
  \item[(b)] $f$ has infinite decomposition ideal, that is, $f$ is not relatively totally transitive;
  \item[(c)] $f$ has a unique periodic point (which in fact is a fixed point);
  \item[(d)] $f$ is almost open but not open;
  \item[(e)] for any non-atomic $f$-invariant measure $\mu$,
     $(X,\BBb_X,f,\mu)$ is a measure-theoretic extension of the dyadic adding machine;
  \item[(f)] there is a compatible convex metric $d$ on $X$ such that $f$ is Lipschitz
     with Lipschitz constant arbitrarily close to $1$;
  \item[(g)] $f$ has positive topological entropy.
\end{enumerate}
\end{theorem}

After we had proved this theorem, the preprint \cite{AHNO13} appeared, where the authors ask
whether every transitive dendrite map admits a terminal periodic decomposition.
Our result gives a negative answer to this question. Let us remark that, by
\cite{AHNO13}, if a transitive dendrite map has infinite decomposition
ideal then the set of transitive points is as in our condition (a), that is,
no transitive point is a cut point of $X$.

The construction of the map $f$ from Theorem~\ref{T:main}
goes as follows (see Section~\ref{S:constr} for details).
We start with the universal dendrite $X$ of order $3$
and the set $Q^*$ of all words of dyadic rationals.
Then, for each word $\alpha\in Q^*$, we choose an arc $A_\alpha$ in $X$.
This is done in such a way that, besides other properties, the union $X^{(m)}$ of all arcs $A_\alpha$
with the length of $\alpha$ at most $m$ is a nowhere dense subdendrite of $X$,
and the union of the increasing sequence $X^{(0)}\subseteq X^{(1)}\subseteq X^{(2)}\subseteq \dots$
is dense in $X$.
Then we inductively define a map $F:X\to X$ on every $X^{(m)}$
by specifying $F$ on the involved arcs $A_\alpha$.
To do that, we use what we call index maps $\tau_\alpha$. Every such
$\tau_\alpha$ maps dyadic rationals $Q$ into $Q$ in a special way, see Lemmas~\ref{L:tau-0} and \ref{L:tau-12}.
On the set $X^{(\infty)}=X\setminus\bigcup_m X^{(m)}$, $F$ is defined
simply as a continuous extension.
The key property of $F$ is that there is a map $\varrho:Q^*\to Q^*$
(defined via index maps $\tau_\alpha$)
such that $F(A_\alpha)\supseteq A_{\varrho(\alpha)}$ for every $\alpha$,
with equality for all $\alpha$ of length at least $2$.
The dynamical properties of $F$ then follows from the corresponding properties of $\varrho$, see
e.g.~Lemmas~\ref{L:F-RPD-rho} and \ref{L:F-trans-Ralpha-implies-trans}.
In Section~\ref{S:F-props} we prove that $F$ has all the properties stated in Theorem~\ref{T:main}
but (c) and (f). Collapsing a subdendrite $X^{(m)}$ to a point gives the desired map $f$
with Lipschitz constant enough close to $1$ when $m$ is sufficiently large;
see Section~\ref{S:proof}.

\section{Preliminaries}\label{S:prelim}
By $\sqcup$ we denote the disjoint union.
The sets of all positive integers and all real numbers are denoted by $\NNN$
and $\RRR$, respectively.
For a bounded interval $I\subseteq\RRR$, $\abs{I}$ denotes the length of $I$.
A map $f:X\to Y$ is called \emph{open} (\emph{almost open}) if $f(U)$
is open (has non-empty interior, respectively) in $Y$ for every non-empty open subset $U$ of $X$.
If $(X,d), (Y,\varrho)$ are metric spaces and $L\ge 0$,
a map $f:X\to Y$ is called \emph{Lipschitz-$L$}
if $\varrho(f(x),f(y))\le L d(x,y)$ for every $x,y\in X$; the smallest such $L$
is the \emph{Lipschitz constant} of $f$ and is denoted by $\Lip(f)$.

\subsection{Continua}
A \emph{continuum} is a compact connected metric space.
A \emph{curve} is a one-dimensional continuum.
An \emph{arc} $A$ is any topological space homeomorphic to the unit interval $[0,1]$; if
$a$ and $b$ are the images of $0$ and $1$, we write $A=[a,b]$.
A \emph{graph} is a continuum which can be written as the union of finitely many
arcs, any two of which are either disjoint or intersect only
in one or both of their end points.
A \emph{tree} is a graph containing no simple closed curve.
A \emph{dendrite} is a (possibly degenerate)
locally connected continuum containing no simple closed curve.
A dendrite is a tree if and only if it is non-degenerate and the set of end points
is finite.

Let $X$ be a continuum.
A point $x\in X$ is called a \emph{cut point} of $X$ if $X\setminus\{x\}$ is disconnected.
For the definition of the \emph{order} of a point $x\in X$, see e.g.~\cite[Definition~9.3]{Nad}.
If $X$ is a continuum, a point $x\in X$ is called an \emph{end (ordinary, branch) point}
of $X$ if its order is one (two, at least three, respectively). The sets
of all end, ordinary and branch points of $X$ are denoted by $\End(X)$, $\Ord(X)$
and $\Branch(X)$. For every dendrite $X$, $\Branch(X)$ is countable
and $\End(X)$ is totally disconnected $G_\delta$ \cite[pp.~278 and 292]{Kur2};
further, if $\Branch(X)$ is dense then $\End(X)$ is dense, hence residual.

A metric $d$ on a continuum $X$ is said to be \emph{convex}
provided for every distinct $x,y\in X$
there is $z\in X$ such that $d(x,z)=d(z,y)=d(x,y)/2$.
If $d$ is convex then for every $a\ne b$ there is an arc
$A$ with end points $a,b$,
the length (that is, the Hausdorff one-dimensional measure)
of which is
equal to $d(a,b)$; every
such arc is called \emph{geodesic}.
Every subarc of a geodesic arc $A=[a,b]$ is geodesic; thus, for every $c\in A$, $d(a,b)=d(a,c)+d(c,b)$.
In dendrites equipped with a convex metric, every arc is geodesic.

Let $X$ be a dendrite and $Y\subseteq X$ be a subdendrite of $X$. A map
$\rho_Y:X\to Y$ is called the \emph{first point retraction} of $X$ onto $Y$
if $\rho_Y(x)=x$ for every $x\in Y$ and, for $x\in X\setminus Y$ and
$C$ being the (connected) component  of $X\setminus Y$ containing $x$,
$f_Y(x)$ is the unique point of the boundary of $C$; see e.g.~\cite[pp.~175--176]{Nad}.

\subsection{Dynamical systems}
A \emph{(discrete) dynamical system} is a pair $(X,f)$ where $X$ is a compact metric
space and $f:X\to X$ is a
continuous map.
A subset $A\subseteq X$ is called \emph{$f$-invariant} (\emph{strongly $f$-invariant})
if $f(A)\subseteq A$ ($f(A)=A$, respectively).
The iterates of $f$ are defined by $f^0=\id_X$ (the identity map on $X$) and
$f^n=f^{n-1}\circ f$ for $n\ge 1$.
A point $x\in X$ is a \emph{fixed (periodic) point} of $f$ if
$f(x)=x$ ($f^n(x)=x$ for some $n\in\NNN$, respectively).
By $\Fix(f)$ and $\Per(f)$ we denote the sets of all fixed and all periodic points of $f$.
The \emph{orbit} of $x$ is the set $\orbit_f(x) = \{f^n(x):\ n\ge 0\}$
and the \emph{trajectory} of $x$ is the sequence $(f^n(x))_{n\ge 0}$.
A \emph{backward trajectory} of $x$ is any sequence $(x_{-n})_{n\ge 0}$ in $X$ such that $x_0=x$ and $f(x_{-n})=x_{-n+1}$
for every $n\ge 1$.

The \emph{topological entropy} of $(X,f)$ will be denoted by $h(f)$.
A dynamical system $(X,f)$ is \emph{(topologically) transitive} if
for every non-empty open sets $U,V\subseteq X$ there is $n\in\NNN$ such that
$f^n(U)\cap V\ne\emptyset$. A point whose orbit is dense is called a \emph{transitive point}.
If $(X,f^n)$ is transitive for all $n\in\NNN$ we say that $(X,f)$ is \emph{totally transitive}.
If $(X\times X, f\times f)$ is transitive then $(X,f)$
is called \emph{(topologically) weakly mixing}.
A system $(X,f)$ is
\emph{(topologically) strongly mixing}
provided for every non-empty open sets $U,V\subseteq X$ there is $n_0\in\NNN$ such that
$f^n(U)\cap V\ne\emptyset$ for every $n\ge n_0$.
A system $(X,f)$ is
\emph{(topologically) exact} or \emph{locally eventually onto} if
for every non-empty open subset $U$ of $X$ there is $n\in\NNN$ such
that $f^n(U)=X$. Further, $(X,f)$ is \emph{Devaney chaotic}
provided $X$ is infinite, $f$ is transitive and has dense set of periodic points.
Finally, a system $(X,f)$ is said to satisfy the \emph{specification property}
if for every $\eps>0$ there is $m$ such that for every $k\ge  2$, for every $k$ points $x_1,\dots,x_k\in X$, for every integers $a_1\le b_1<\dots<a_k\le b_k$ with
$a_i-b_{i-1}\ge m$ ($i=2,\dots,k$) and for every integer $p\ge m+b_k-a_1$, there
is a point $x\in X$ with $f^p(x)=x$ such that
$d(f^n(x),f^n(x_i))\le \eps$ for $a_i\le n\le b_i, 1\le i\le k$.

\subsection{Regular periodic decompositions}
Let $X$ be a compact metric space. A set $D\subseteq
X$ is \emph{regular closed} if it is the closure of its
interior or, equivalently, if it equals the closure of an
open set.

Let $f:X\to X$ be a continuous map on a compact metric space. By \cite{Banks},
a \emph{regular
periodic decomposition} for $f$ is a finite sequence
$\mathcal D=(D_0,\dots ,D_{m-1})$ of $m\ge 1$ regular closed subsets
of $X$ covering $X$ such that $f(D_i)\subseteq D_{i+1 \modm{m}}$ for
$0\leq i<m$ and $D_i\cap D_j$ is nowhere dense in $X$
for $i\ne j$.
The integer
$m$ is called the \emph{length} of $\mathcal D$.
The set of the lengths of all regular periodic decompositions of $f$
is called the \emph{decomposition ideal} of $f$ and is denoted by $DI(f)$;
always $1\in DI(f)$.

Let $f:X\to X$ be a transitive map.
If $f^n$ is non-transitive for some
$n\geq 2$, then $f$ has a regular periodic decomposition
of length $m\ge 2$ dividing $n$ \cite[Corollary~2.1]{Banks}.
Thus, $DI(f)$ equals $\{1\}$ if and only if $f$ is totally transitive.
If $\mathcal D=(D_0,\dots ,D_{m-1})$ is a regular periodic decomposition for $f$
then, for every $i$, $D_i$ is strongly $f^m$-invariant
and $f^m|_{D_i}$ is transitive \cite[Theorem~2.1]{Banks}.

Following \cite{Banks}, we say that a map $f:X\to X$ is
\emph{relatively totally transitive (relatively mixing)} if there is a regular periodic decomposition
$\mathcal D=(D_0,\dots ,D_{m-1})$ for $f$ such that $f^m|_{D_i}$ is totally transitive (mixing)
for every $i$. Analogously we define the relative specification property.
Note that a transitive map is not relatively totally transitive if and only if
its decomposition ideal is infinite.
\emph{Adding machines} are fundamental examples of such systems;
for the definition, see e.g.~\cite[p.~520]{Banks}.

\subsection{Auxiliary index maps}
Here we construct auxiliary maps
which will be used in the proof of Theorem~\ref{T:main}, see Section~\ref{S:constr}.
We will need two types of these `index' maps. The first type (see Lemma~\ref{L:tau-0}) will be used
in the definition of $F(b_r)$ ($r\in Q$). The other type (see Lemma~\ref{L:tau-12}) will be
used for the definition of $F|_{A_r}$ ($r\in Q$)
and $F|_{A_\alpha}$ ($\abs{\alpha}\ge 2$), see (\ref{EQ:def-F-X1}) and (\ref{EQ:def-F-X2}).
We say that a function
$g:Z\to\RRR$ defined on a non-empty set $Z$,
or an indexed set $\{g(z)\}_{z\in Z}\subseteq \RRR$, is \emph{null} and
we write
$
 \lim_{z\in Z} g(z)=0
$
if for every $\eps>0$ there are at most finitely many points
$z\in Z$ with $\abs{g(z)}>\eps$.

\begin{lemma}\label{L:tau-0}
 Let $\varphi:[0,1]\to [0,1]$ be a continuous surjection
 with $\varphi(t)<t$ for $t\in (0,1)$.
 Let $p\ge 2$ be an integer and $R_c$ ($0\le c<p$)
 be pairwise disjoint countable dense subsets of $(0,1)$; put $R=\bigcup_c R_c$.
 Then there is a map $\xi:R\to R$ and a two-sided sequence $(z_m)_{m\in \ZZZ}$ in $R$ such that
 \begin{enumerate}
   \item[(a)] $\xi$ is a surjection, which is one-to-one outside of a countable set, on which it is two-to-one;
   \item[(b)] $\xi(R_c)=R_{c+1 \modm{p}}$ for every $0\le c<p$;
   \item[(c)] $\varphi(r) < \xi(r)$ for every $r\in R$ and $\lim\limits_{r\in R} (\xi(r)- \varphi(r))=0$;
   \item[(d)] $\xi(z_m)=z_{m+1}>z_m>\varphi(z_{m+1})$ for every $m\in\ZZZ$, $\lim\limits_{m\to -\infty} z_m=0$ and $\lim\limits_{m\to \infty} z_m=1$;
   \item[(e)] for every $r\in R$ there are $n_r\ge 0$ and $m_r\in\ZZZ$ with
     $$
      r>\xi(r)>\dots > \xi^{n_r-1}(r) \ < \ \xi^{n_r}(r)=z_{m_r} \ < \ \xi^{n_r+1}(r)=z_{m_r+1} < \dots
     $$
     and $\xi^i(r)\not\in\{z_m:\ m\in\ZZZ\}$ for every $0\le i<n_r$;
   \item[(f)] $\lim_{n\to\infty} \xi^n(r)=1$ for every $r\in R$; moreover, every $r\in R$ has a backward trajectory converging to $1$;
   \item[(g)] for every $c,d<p$, $r\in R_c$ and $s\in R$ there are $h,k\in\NNN_0$ and $t\in R_{d}\cap (\varphi(r'),\xi(r'))$
     (where $r'=\xi^k(r)$) with $\xi^h(t)=s$.
 \end{enumerate}
\end{lemma}
\begin{proof}
For every $k\in\ZZZ$ we define $R_k=R_{k\modm{p}}$.
Note that $\lim_{n\to\infty} \varphi^n(t)=0$ for every $t<1$.
Thus we can fix a two-sided sequence $(z_m)_{m\in\ZZZ}$ from $R$ such that, for every $m\in\ZZZ$,
\begin{equation}\label{EQ:tau-0-rm-choice}
    z_m\in R_m, \qquad
    \varphi(z_m)<z_{m-1}<z_m, \qquad
    \lim_{m\to -\infty} z_m=0
    \qquad\text{and}\qquad
    \lim_{m\to \infty} z_m=1.
\end{equation}
For every $m\in\ZZZ$ put
\begin{equation}\label{EQ:tau-0-rm}
 \xi(z_m) = z_{m+1}.
\end{equation}
Put $S=R\setminus\{z_m:\ m\in \ZZZ\}=\{s_1,s_2,\dots\}$, $K^{(0)}=\emptyset$
and fix a sequence $(\eps_j)_{j\in\NNN}$ of positive reals decreasing to zero.

We construct $\xi$ by induction.
Assume that $j\ge 1$ and that $\xi$ has been defined for every $s_k$ with $k\in K^{(j-1)}$, where $\{s_k: k\in K^{(j-1)}\}$ is either
empty or has a unique accumulation point $1$.
Let $k_{j,0}$ be the smallest integer from (an infinite set) $\NNN\setminus K^{(j-1)}$.
Find integers $m_j\in\ZZZ$ and $l_j\ge 1$ such that
\begin{enumerate}
  \item[($A_j$)] $\varphi^{l_j}(s_{k_{j,0}}) < z_{m_j} < \eps_j$ and, if $j>1$, $m_j<m_{j-1}$;
  \item[($B_j$)] $z_{m_j}\in R_{c_j+l_j+1}$, where $c_j<p$ is such that $s_{k_{j,0}}\in R_{c_j}$.
\end{enumerate}
Since every $R_c$ is countable dense in $(0,1)$ and $\varphi$ is continuous with $\varphi(1)=1$ and $\varphi(t)<t$ for $t\in (0,1)$, there are
pairwise distinct integers $k_{j,i}$ ($-\infty<i\le l_j$, $i\ne 0$) from $\NNN\setminus (K^{(j-1)}\cup\{k_{j,0}\})$ such that
\begin{enumerate}
  \item[($C_j$)] $\dots > s_{k_{j,-2}} > s_{k_{j,-1}} > s_{k_{j,0}} > s_{k_{j,1}} > \dots  > s_{k_{j,l_j}} \ < z_{m_j}$;
  \item[($D_j$)] $s_{k_{j,i+1}} \in (\varphi(s_{k_{j,i}}), \varphi(s_{k_{j,i}})+\eps_j)\cap R_{c_j+i+1}$ for every $-\infty<i<l_j$;
  \item[($E_j$)] $\lim_{i\to -\infty} s_{k_{j,i}}=1$; moreover, for every sufficiently large (positive) $m$,
                 $[z_m,z_{m+1})$ contains at least $p$ points $s_{k_{j,i}}$.
\end{enumerate}
Define
\begin{equation}\label{EQ:tau-0-sj}
 \xi(s_{k_{j,i}}) = s_{k_{j,i+1}} \quad \text{for } -\infty<i<l_j,\qquad
 \xi(s_{k_{j,l_j}})=z_{m_j}
\end{equation}
and put $K^{(j)}=K^{(j-1)}\cup\{k_{j,i}:\ -\infty < i\le l_j\}$.

After finishing the induction we obtain a map $\xi$ defined on $R$. We prove that $\xi$ has the required properties.

(a) This follows from the facts that every $r\in R\setminus\{z_{m_1},z_{m_2},\dots\}$ has exactly one preimage and every $z_{m_j}$ has exactly
two preimages (use ($A_j$)).

(b) By ($D_j$), ($B_j$) and (\ref{EQ:tau-0-rm-choice}), $\xi(R_c)\subseteq R_{c+1}$ for every $c<p$.
By the choice of integers $k_{j,0}$, the equalities hold.

(c) Fix $r\in R$. If $r=z_m$ for some $m$ then $\xi(r)=z_{m+1}>r>\varphi(r)$ by (\ref{EQ:tau-0-rm}) and (\ref{EQ:tau-0-rm-choice}).
If $r=s_{k_{j,l_j}}$ for some $j\ge 1$, then $\xi(r)=z_{m_j}>r>\varphi(r)$ by ($C_j$)
and (\ref{EQ:tau-0-sj}). Otherwise $r=s_{k_{j,i}}$ for some $j\ge 1$
and $i<l_j$; then $\xi(r)=s_{k_{j,i+1}}>\varphi(r)$ by ($D_j$) and  (\ref{EQ:tau-0-sj}). Hence $\xi(r)>\varphi(r)$ for every $r\in R$.
Now we prove that $\lim_{r\in R} (\xi(r)- \varphi(r))=0$. To this end, fix $\eps>0$ and put $T_\eps=\{r\in R:\ \xi(r)-\varphi(r)>\eps\}$.
By (\ref{EQ:tau-0-rm-choice}), continuity of $\varphi$ and the fact that $0,1$ are fixed points of $\varphi$,
the intersection $T_\eps\cap\{z_m:m\in\ZZZ\}$ is finite.
Since $\eps_j\searrow 0$, there is $j_0$ with $\eps_{j_0}\le \eps$. Then $T_\eps\cap S \subseteq \{s_k:\ k\in K^{(j_0)}\}$ by ($D_j$) for $j>j_0$.
Now ($E_j$) for $j\le j_0$,
$\varphi(1)=1$ and continuity of $\varphi$ at $1$ give that $T_\eps\cap S$ is finite. Thus $T_\eps$ is finite and (c) is proved.

(d) This immediately follows from (\ref{EQ:tau-0-rm-choice}) and (\ref{EQ:tau-0-rm}).

(e) If $r=z_m$ for some $m$ then put $n_r=0$, $m_r=m$. Otherwise there are $j$ and $i\le l_j$ with $r=s_{k_{j,i}}$; in such a case put
$n_r=l_j-i+1$, $m_r=m_j$ and use (\ref{EQ:tau-0-sj}).

(f) The first part of (f) follows from (d) and (e). The second part is implied by
 $\lim_{j\to\infty} m_j= -\infty$ (use ($A_j$)) and $\lim_{i\to -\infty} s_{k_{j,i}}=1$ for every $j$ (use ($E_j$)).

(g) Take any $c,d<p$, $r\in R_c$ and $s\in R$. By (f) there is a backward trajectory $(t_{-l})_{l\ge 0}$ of $s$ with $\lim_l t_{-l}=1$.
By (\ref{EQ:tau-0-rm}), the second part of ($E_j$) and (b),
there are $h$ and $m\ge m_r$ with $t_{-h}\in [z_m,z_{m+1})\cap R_{d}$. If we put $k=n_r + (m-m_r)$,
$r'=\xi^k(r)$ and $t=t_{-h}$, then
$r'=z_m$, $s=\xi^h(t)$ and $(\varphi(r'),\xi(r'))\supseteq [z_m,z_{m+1})\ni t$.
\end{proof}

\begin{lemma}\label{L:tau-12}
 Let $I=[a,b]$, $I'=[a',b']$ be non-degenerate compact intervals.
 Let $C\ne\emptyset$ be a countable index set and
 $\{R_c\}_{c\in C},\{R'_c\}_{c\in C}$ be families of dense countable
 subsets of $I,I'$ such that the sets $R_c$ are pairwise disjoint.
 Let $L>\abs{I'}/\abs{I}$.
 Then there is a map $\tau:I\to I'$ such that
 \begin{enumerate}
   \item[(a)] $\tau$ is a Lipschitz-$L$ surjection; moreover, if the sets $R'_c$ are pairwise disjoint, $\tau$ is a homeomorphism;
   \item[(b)] $\tau(a)=a'$, $\tau(b)=b'$, $\tau((a,b)) = (a',b')$;
   \item[(c)] $\tau(R_c)=R'_{c}$ for every $c$;
   \item[(d)] if $C$ is finite then, for every $c$, $\tau^{-1}(R_c')\subseteq \bigcup_d R_d$ and
     every $r'\in R_c'$ has only finitely many preimages;
   \item[(e)] $\tau(t)\le(\abs{I'}/\abs{I}) (t-a)+a'$ for every $t\in I$.
 \end{enumerate}
\end{lemma}

The map $\tau$ from Lemma~\ref{L:tau-12} will be the limit of a sequence of piecewise linear maps,
obtained using the following lemma. We say that a map $g:I\to I'$ (where $I=[a,b],I'=[a',b']$
are compact intervals) is \emph{piecewise linear} with respect to the set $K=\{a_0=a<a_1<\dots<a_k=b\}$,
or \emph{$K$-linear}, if the restrictions $g|_{[a_{i},a_{i+1}]}$ ($0\le i<k$) are linear
(thus $g$ is necessarily continuous);
the intervals $[a_{i}, a_{i+1}]$ will be called \emph{$K$-intervals}.
If $g$ additionally satisfies $g((a_i,a_{i+1}))\cap g(K)=\emptyset$ for every $i<k$ such that
$g$ is not constant on $[a_i,a_{i+1}]$ (that is, $g(a_i)\ne g(a_{i+1})$), then $g$ is said to be
\emph{strongly piecewise linear} with respect to $K$, or \emph{strongly $K$-linear}.
Note that if a $K$-linear map $g$ is strictly monotone then it is strongly $K$-linear.
For a $K$-linear map $g:I\to I'$ and points $r_1,\dots,r_k\in I\setminus K$, $r'\in I'$
let $g^{[r_1,\dots,r_k;r']}:I\to I'$ denote the $\tilde{K}$-linear map (where $\tilde{K}=K\cup\{r_1,\dots,r_k\}$)
which maps every $r_i$ to $r'$ and every $x\in K$ to $g(x)$.
By $\|g\|$ we denote the supremum norm of $g$.

\begin{lemma}\label{L:tau-12-step}
Let $g:[0,1]\to [0,1]$ be a strongly $K$-linear surjection with $\Lip(g)<L$
and $0<g(t)<t$ on $(0,1)$.
Let $R,R'$ be countable dense subsets of $(0,1)$ such that $R\cap K=\emptyset$ and let $\eps>0$.
Then
\begin{enumerate}
  \item[(a)] for every $r'\in R'$ there are $k\ge 1$ and $r_1,\dots,r_k\in R$, and
  \item[(b)] for every $r_1\in R$ there are $r'\in R'$, $k\ge 1$ and $r_2,\dots,r_k\in R\setminus\{r_1\}$,
\end{enumerate}
such that the map $\tilde{g}=g^{[r_1,\dots,r_k;r']}$ is a strongly $\tilde{K}$-linear surjection (where
$\tilde{K}=K\cup\{r_1,\dots,r_k\}$) and
\begin{equation}\label{EQ:tau-12-step}
    \|\tilde{g}- g\|<\eps, \qquad
    \Lip(\tilde{g})<L
    \qquad\text{and}\qquad
    0<\tilde{g}(t)<t \text{ on } (0,1).
\end{equation}
Moreover, if $g$ is a homeomorphism and, in (a), $r'\not\in g(K)$, then in both cases
we may additionally assume that $k=1$ and $\tilde{g}$ is a homeomorphism.
\end{lemma}
\begin{proof}
Note that if every $r_i$ is such that $g(r_i)$ is sufficiently close to $r'$,
then automatically $r_i>r'$ for every $i$ and (\ref{EQ:tau-12-step}) is satisfied.
Thus we only need to ensure that $\tilde{g}$ is
\emph{strongly} $\tilde K$-linear.
This can be done as follows. Let $K=\{a_0=0<a_1<\dots<a_m=1\}$ and $a_i'=g(a_i)$.
In (a), let $N$ denotes the set of those indices $0\le i<k$ for which $r'\in(a_i',a_{i+1}')$.
If $N$ is empty, there is $1\le i<k$ with $a_i'=r'$; put $k=1$ and take arbitrary
$r_1\in R\cap (a_i,a_{i+1})$ which is sufficiently close to $a_i$.
Otherwise $N=\{i_1,\dots,i_k\}$ is non-empty; then in every $(a_{i_j}, a_{i_j+1})$ take some
$r_j\in R$ for which $g(r_j)$ is sufficiently close to $r'$.

In (b), let $i$ be such that $r_1\in (a_i,a_{i+1})$.
If $a_i'\ne a_{i+1}'$
take some $r'\in R'\cap(a_i',a_{i+1}')$ which is sufficiently close to $g(r_1)$;
then, provided there are indices $j\ne i$ with $r'\in (a_j,a_{j+1})$,
in every such interval $(a_j,a_{j+1})$ choose some point from $R$ as in (a).
On the other hand, if $a_i'=a_{i+1}'$ (that is, $g$ is constant on $[a_i,a_{i+1}]$),
then $a_i'>0$ and we can
take $r'\in R'$, $r'<a_i'$ sufficiently close to $a_i'=g(r_1)$ and such that
$[r',a_i')\cap g(K)=\emptyset$; then  we define $k$ and $r_2,\dots,r_k$ as in the case
$a_i'\ne a_{i+1}'$.

In all the cases, the resulting map $\tilde g = g^{[r_1,\dots,r_k;r']}$
is a strongly $\tilde K$-linear surjection (with $\tilde K=K\cup \{r_1,\dots,r_k\}$) satisfying
(\ref{EQ:tau-12-step}).
To finish the proof it suffices to note that if $g$ is a homeomorphism and, in (a), $r'\not\in g(K)$,
then in the previous construction we always obtain that $k=1$ and
$\tilde{g}$ is strictly increasing.
\end{proof}

\begin{proof}[Proof of Lemma~\ref{L:tau-12}]
We may assume that $I=I'=[0,1]$.
Take a sequence $(\eps_j)_{j\ge 0}$ of positive reals such that $\sum_j \eps_j<\infty$.
Enumerate the elements of every $R_c$ and $R'_c$.
Let $(c_j)_{j\ge 0}$ be a sequence from $C$ containing every $c\in C$ infinitely many times.

The map $\tau$ will be the uniform limit of strongly $K_j$-linear maps $\tau_j:I\to I'$ ($j\ge 0$), where
$K_0\subseteq K_1\subseteq\dots$.
First,
take any $r\in R_{c_0}$ and, sufficiently close to it, $r'\in R_{c_0}'$ with $r'<r$. Put $K_0=\{a,r,b\}$ and
denote by $\tau_0:I\to I'$ the (strongly) $K_0$-linear surjection mapping $0,r,1$ to $0,r',1$, respectively.
Note that $\Lip(\tau_0)<L$ and $0<\tau_0(t)<t$ on $(0,1)$ by the choice of $r'$.
Put $S_{c_0,0}=\{r\}$, $S_{c_0,0}'=\{r'\}$ and, for every $c\ne c_0$, $S_{c,0}=S'_{c,0}=\emptyset$.

Assume now that $j\ge 1$ and that a
$K_{j-1}$-linear map $\tau_{j-1}$ with $\Lip(\tau_{j-1})<L$ has been defined, together
with finite sets $S_{d,j-1}\subseteq R_d$, $S'_{d,j-1}\subseteq R_d'$ ($d\in C$) such that
$K_{j-1}=\{0,1\}\sqcup\bigsqcup_d S_{d,j-1}$ and $\tau_{j-1}(K_{j-1}) = \{0,1\}\cup\bigcup_d S'_{d,j-1}$.
Put $c=c_j$ and define $S_{d,j}=S_{d,j-1}$, $S'_{d,j}=S'_{d,j-1}$ for every $d\in C\setminus\{c\}$.
Let $r$ be the first (according to the enumeration of $R_c$)
element of $R_c\setminus S_{c,j-1}$. Lemma~\ref{L:tau-12-step}(b), applied to $g=\tau_{j-1}$,
$K=K_{j-1}$, $R=R_c\setminus S_{c,j-1}$, $R'=R'_c\setminus S'_{c,j-1}$, $\eps=\eps_j$ and $r_1=r$,
gives a map $\tilde\tau_j=\tau_{j-1}^{[r_1,\dots,r_k;r']}$.
Let $s'$ be the first element of  $R_c'\setminus (S_{c,j-1}'\cup\{r'\})$. By applying Lemma~\ref{L:tau-12-step}(a)
to $g=\tilde\tau_j$, $K=K_{j-1}\cup\{r_1,\dots,r_k\}$, $R=R_c\setminus(S_{c,j-1}\cup\{r_1,\dots,r_k\})$,
$R'=R_c'\setminus (S_{c,j-1}'\cup\{r'\})$,
$\eps=\eps_j$ and $r'=s'$, we obtain
a map $\tau_j=\tilde\tau_j^{[s_1,\dots,s_l;s']}$.
To finish the induction step, we put
$S_{c,j}=S_{c,j-1}\cup\{r_1,\dots,r_k,s_1,\dots,s_l\}$, $S_{c,j}'=S_{c,j-1}'\cup\{r',s'\}$,
and $K_j=\{0,1\}\sqcup\bigsqcup_d S_{d,j}=K_{j-1}\sqcup\{r_1,\dots,r_k,s_1,\dots,s_l\}$.

For every $j$ we have $\|\tau_j-\tau_{j-1}\|<2\eps_j$   by (\ref{EQ:tau-12-step}).
Since $\sum_j\eps_j<\infty$, we may define $\tau=\lim_j \tau_j$.
Immediately $\tau(a)=a'$, $\tau(b)=b'$ and $\tau$ is continuous, thus it is also surjective.
Further, $\Lip(\tau)\le\limsup_j \Lip(\tau_j)\le L$ by (\ref{EQ:tau-12-step}).
If the sets $R_c'$ are pairwise disjoint, then, by Lemma~\ref{L:tau-12-step}
and the fact that $\tau_0$ is strictly increasing,
we may assume that every $\tau_j$ is strictly increasing; thus $\tau$ is increasing.
Since disjointness of the sets $R_c'$
and denseness of the sets $R_c$
immediately imply that $\tau$ has no interval of constancy, $\tau$ is a homeomorphism. Hence (a) and (b) are proved.

The property (c) follows from the construction since $R_c=\bigcup_j S_{c,j}$ and $R_c'=\bigcup_j S_{c,j}'$
for every $c\in C$.
To show (e) it suffices to realize that $\tau_j(t)<t$ for every $j$ and $t\in (0,1)$,
thus $\tau(t)\le t$ on $(0,1)$.

It remains to prove (d).
To this end, assume that $C$ is finite and take any $r'\in\bigcup_c R'_c$; put $C_0=\{c\in C: r'\in R_c'\}$.
Since $C_0$ is finite, there is $j_0$ such that $r'\in S'_{c,j}$ for every $c\in C_0,j\ge j_0$.
Assume, contrary to (d), that there is $r\in (0,1)\setminus\left( \bigcup_c R_c \right)$ with $\tau(r)=r'$.
Then for every $j\ge j_0$ there is a $K_j$-interval $[b_j,c_j]$ such that $r\in (b_j,c_j)$.
Strong $K_j$-linearity of $\tau_j$ and the fact $r'\in \tau_j((b_j,c_j))\cap g(K_j)$
give that $\tau_j$ is constant on $[b_j,c_j]$ and thus $b_j,c_j$ are $\tau_j$-preimages of $r'$.
Since, by the construction and finiteness of $C_0$, the set $\bigcup_j \tau_j^{-1}(r')$ is finite,
we obtain that the maps $\tau_j$ ($j\ge j_0$), and thus also $\tau$, have a common
non-degenerate interval of constancy. But this is clearly impossible.
So we have that $\tau^{-1}(r')\subseteq \bigcup_c R_c$. Since, by the construction,
$r'$ has only finitely many preimages in any $R_c$, the set $\tau^{-1}(r')$ is finite.
\end{proof}

\section{Construction of the map $F$}\label{S:constr}

The purpose of this section is to construct a dendrite $X$ and a transitive map
$F:X\to X$. Then, in Section~\ref{S:F-props},
we prove that this map has all the properties stated in Theorem~\ref{T:main}
but (c) and (f), since it has two fixed points and need not be Lipschitz.
We start with some notation and description of the phase space $X$, the universal
dendrite of order $3$.

\subsection{Basic notation}
By $Q$ we denote the set of all dyadic rational numbers in $(0,1)$.
Every $r\in Q$ can be uniquely written as $r=p_r/2^{q_r}$ with $q_r\ge 1$
and odd $1\le p_r<2^{q_r}$.
By $Q^0$, $Q^*$ and $Q^+$ we mean the sets $\{\emptyword\}$, $\bigcup_{k\ge 0} Q^k$ and
$\bigcup_{k\ge 1} Q^k$, respectively;
here and below, the symbol $\emptyword$ denotes the empty word.

By $\abs{\alpha}$ we denote the length of $\alpha\in Q^*$, that is, an integer $k$ such that $\alpha\in Q^k$.
For $\alpha\in Q^k, \beta\in Q^m$ ($k,m\ge 0$)
we naturally define the concatenation $\alpha\beta\in Q^{k+m}$;
note that $\alpha \emptyword = \emptyword\alpha=\alpha$.
If $\alpha=r_0\dots r_{k-2} r_{k-1}\in Q^k$ with $k\ge 1$,
then $\tilde\alpha$ denotes $r_0\dots r_{k-2}$ and,
for $\alpha=\emptyword$, we put $\tilde\alpha=\emptyword$.
Further, for every $0\le m\le k$ we put $\alpha^{(m)}=r_0\dots r_{m-2} r_{m-1}$
(with the interpretation $\alpha^{(0)}=\emptyword$);
so, for example, $\tilde{\alpha}=\alpha^{(\abs{\alpha}-1)}$ for $\alpha\in Q^+$.

Put $\Sigma=\{0,1\}$, $\Sigma^0=\{\emptyword\}$ and define $\Sigma^*$, $\Sigma^+$ analogously as above.
On $\Sigma^m$ ($m\ge 1$) we define
the addition in the usual way with carry to the right and we identify $1$ with the word
$10^{m-1}$; so $\gamma+n$ is defined for any $\gamma\in\Sigma^m$ and $n\in\ZZZ$.
Analogously for $\Sigma^\infty$.
The \emph{longest common prefix} of  $\gamma,\gamma'\in \Sigma^*$ is the word $\delta\in \Sigma^*$ such that
we can write $\gamma=\delta\overline{\gamma}$, $\gamma'=\delta\overline{\gamma}'$ with either one of
$\overline{\gamma},\overline{\gamma}'$ being empty or $\overline{\gamma}_0\ne\overline{\gamma}'_0$.

For $\gamma=\gamma_0\dots\gamma_{m-1}\in\Sigma^m$ ($m\ge 0$) we define
\begin{equation}\label{EQ:Q-gamma}
 Q_\gamma = \{\alpha=\alpha_0\dots\alpha_{m-1}\in Q^m:\ q_{\alpha_i}\equiv \gamma_i\modm{2} \quad \text{for } 0\le i<m \}.
\end{equation}
Then $Q_\emptyword=\{\emptyword\}$, $Q_\gamma=\prod_{i<m} Q_{\gamma_i}$ is dense in $(0,1)^m$ for every $\gamma\in\Sigma^m$ and, for every $m\ge 0$,
$$
 Q^m=\bigsqcup_{\gamma\in\Sigma^m} Q_\gamma.
$$
Analogously we define $Q_\gamma$ for $\gamma\in\Sigma^\infty$.

\medskip

\subsection{Dendrite $X$}
Let $X$ be the universal dendrite of order $3$, that is, the topologically unique dendrite
such that the branch points of it are dense and all have order $3$. We may write
$$
 X
 = \closure{\bigcup_{m\ge 0} X^{(m)}}
 = \bigcup_{m\ge 0} X^{(m)} \ \sqcup \ X^{(\infty)},
$$
where $X^{(0)}=[a_\emptyword,b_\emptyword]$ is an arc and every $X^{(m)}$ ($m\ge 1$) satisfies
$$
 X^{(m)} =
 X^{(m-1)} \sqcup
 \left(
 \bigsqcup_{\alpha\in Q^m} (a_\alpha,b_\alpha]
 \right)
 \qquad
 \text{with }a_\alpha\in (a_{\tilde\alpha},b_{\tilde\alpha})
 \text{ for } \alpha\in Q^m;
$$
moreover,
$$
 \Branch(X) =
   \{a_\alpha:\ \alpha\in Q^+\}
 \qquad\text{and}\qquad
 \End(X) =
   \{a_\emptyword, b_\emptyword\}
   \ \sqcup \ \{b_\alpha:\ \alpha\in Q^+\}
   \ \sqcup \ X^{(\infty)}.
$$
For every $\alpha\in Q^*$ we put $A_\alpha=[a_\alpha,b_\alpha]$ and
define a natural ordering on $A_\alpha$ such that
$a_\alpha<b_\alpha$; we may assume that $a_{\alpha r}<a_{\alpha s}$ for every $r<s$ from $Q$.
The ordinary points of $A_\alpha$ are denoted by $a_{\alpha t}$ ($t\in (0,1)\setminus Q$) in such a way that
for every $t<u$ from $(0,1)$
we have $a_{\alpha}<a_{\alpha t}<a_{\alpha u}<b_\alpha$; we also put $a_{\alpha 0}=a_\alpha$, $a_{\alpha 1}=b_\alpha$.
Note that, for every $m\ge 0$, $X^{(m)}$ is a nowhere dense subdendrite of both $X^{(m+1)}$ and $X$,
and $\bigcup_m X^{(m)}$ is
a dense connected subset of $X$ of the first category. Further,
$X^{(\infty)}$ is a totally disconnected dense $G_\delta$ subset of $X$.

For $\alpha\in Q^*$ let $X_\alpha$ denote the closure of the component of $X\setminus\{a_\alpha\}$
containing $b_\alpha$. So $X_\emptyword=X$ and, for every $\alpha\in Q^*$ and $m\ge 1$,
\begin{equation}\label{EQ:X-alpha}
  X_\alpha=
     \left(  \bigcup_{0\le\abs{\beta}<m} A_{\alpha\beta}   \right)
     \cup
     \left(  \bigsqcup_{\beta\in Q^m} X_{\alpha\beta}   \right)
  =
     \left(  \bigcup_{\beta\in Q^*} A_{\alpha\beta}   \right)
     \cup
     \{ b_{\alpha\beta}:\ \beta\in Q^\infty  \},
\end{equation}
where, for any $\nu=\nu_0\nu_1\ldots \in Q^\infty$, $b_\nu$
denotes the unique point of $\bigcap_m X_{\nu_0\dots\nu_{m-1}}$.
Note that
$$
  X^{(\infty)}=\{b_\alpha:\ \alpha\in Q^\infty\}.
$$

\subsection{Covers $\DDd_m$ of $X$}
Here we define a refining sequence
of finite covers $\DDd_m=(D_\gamma)_{\gamma\in\Sigma^m}$ ($m\ge 0$) of $X$.
For any $m\ge 0$ and $\gamma=\gamma_0\dots\gamma_{m-1}\in\Sigma^m$ put
(recall the definition (\ref{EQ:Q-gamma}) of $Q_\gamma$)
$$
 Y^{(\gamma)}
 = \bigcup_{k=0}^m
 \left(
   \bigsqcup_{\alpha\in Q_{\gamma_0\dots\gamma_{k-1}}} A_\alpha
 \right)
 = A_\emptyword
   \sqcup
   \left(
     \bigsqcup_{k=1}^m \bigsqcup_{\alpha\in Q_{\gamma_0\dots\gamma_{k-1}}} (a_\alpha,b_\alpha]
   \right);
$$
this is a subdendrite of $X^{(m)}$ and $Y^{(\emptyword)}=A_\emptyword$,
$Y^{(\gamma c)} = Y^{(\gamma)} \cup \left( \bigsqcup_{\alpha\in Q_{\gamma c}} A_\alpha \right)$
for every $\gamma\in\Sigma^*$, $c\in\Sigma$.
Now define
\begin{equation}\label{EQ:RPD-Dgamma-def}
 D_{\gamma} = Y^{(\tilde\gamma)} \cup \bigsqcup_{\alpha\in Q_\gamma} X_\alpha
\end{equation}
and
\begin{equation}\label{EQ:RPD-Dm-def}
 \DDd_m=(D_{\gamma})_{\gamma\in\Sigma^m}.
\end{equation}
Particularly, for $c,d\in\Sigma$ we have
$$
 D_\emptyword=X, \qquad
 D_c=A_\emptyword \cup
   \left(
     \bigsqcup_{r\in Q_{c}} X_r
   \right), \qquad
 D_{cd}=A_\emptyword
   \cup
   \left(
     \bigsqcup_{r\in Q_{c}} A_r
   \right)
   \cup
   \left(
     \bigsqcup_{rs\in Q_{cd}} X_{rs}
   \right).
$$
The following lemma, the easy proof of which is skipped, summarizes properties of
the covers $\DDd_m$.

\begin{lemma}\label{L:def-RPD}
 For every $m,k\ge 1$ and different $\gamma,\delta\in\Sigma^m$ the following hold:
 \begin{enumerate}
   \item[(a)] $D_\gamma$ is a regular closed subdendrite with
     $$
      \interior{D_\gamma} = \bigsqcup_{\alpha\in Q_\gamma} (X_\alpha\setminus\{a_\alpha\})
      \qquad\text{and}\qquad
      \boundary{D_\gamma} = Y^{(\tilde\gamma)};
     $$
   \item[(b)] $D_\gamma=\bigcup_{\gamma'\in \Sigma^k} D_{\gamma\gamma'}$;
   \item[(c)] $D_\gamma\cap D_{\delta} = Y^{(\tilde\gamma)}\cap Y^{(\tilde{\delta})} = Y^{(\epsilon)}$,
     where $\epsilon$ is the longest common
     prefix of $\tilde\gamma,\tilde{\delta}$;
   \item[(d)] $\DDd_{m+k}$ refines $\DDd_m$;
   \item[(e)] $\bigcup_{\gamma\ne\delta} (D_\gamma\cap D_{\delta}) = X^{(m-1)}$
     and $\bigcap_{\gamma\in\Sigma^m} D_\gamma = X^{(0)}$.
 \end{enumerate}
\end{lemma}

We identify words $\gamma\in\Sigma^m$ with integers $0\le k<2^m$ in such a way that
the word $0^m+k$ is identified with $k$. Thus it has a meaning to say that $\DDd_m$
is a regular periodic decomposition of some continuous selfmap of $X$.

\subsection{Definition of the map $F$}\label{SS:F-def}
In this section we inductively construct a selfmap $F$ of $X$
together with auxiliary maps $\tau_\alpha$ ($\alpha\in Q^*$).
Fix a sequence $(L_m)_{m\ge 1}$ of reals such that
\begin{equation}\label{EQ:Lm-def}
    L_m\searrow 1
    \qquad\text{and}\qquad
    \sum_m \left( \prod_{i\le m}L_i^{-1} \right) <\infty
\end{equation}
(take, for example, $L_m=(1+1/m)^2$).
For $\gamma\in \Sigma^+$ let $\last{\gamma}\in\Sigma$
denote the last letter of $(\gamma+1)$; that is, for $\gamma\in\Sigma^m$,
$\last\gamma=\gamma_{m-1}+1$ if $\gamma_i=1$ for every $i<m-1$, and $\last\gamma=\gamma_{m-1}$ otherwise.
We will often use that, for any $\gamma\in \Sigma^+$ and $c\in\Sigma$,
\begin{equation}\label{EQ:e-gamma}
  (\gamma+1) \last{\gamma c} = \gamma c+1
  \qquad\text{and}\qquad
  \last{\gamma 0}\ne \last{\gamma 1}.
\end{equation}
To shorten the notation we put $\last\alpha=\last\gamma$ for every $\alpha\in Q_\gamma$.

First we define $F$ on $X^{(0)}$. Take an
increasing homeomorphism $\varphi_\emptyword:[0,1]\to[0,1]$
such that $\varphi_\emptyword(t)<t$ for every $t\in (0,1)$ and $\varphi_\emptyword(Q_c)=Q_{c+1}$
for every $c\in \Sigma$; for example, one can use the map
$\varphi_\emptyword$ given by $\varphi_\emptyword(t)=t/2$ for $t\le 2/3$ and
$\varphi_\emptyword(t)=2t-1$ for $t\ge 2/3$. Now define $F$ on $X^{(0)}$ by
\begin{equation}\label{EQ:def-F-X0}
 F(a_t)=a_{\varphi_\emptyword(t)}
 \qquad\text{for every } t\in [0,1].
\end{equation}
Let $\tau_\emptyword:Q\to Q$ be the map $\xi$ obtained from Lemma~\ref{L:tau-0}
for $\varphi=\varphi_\emptyword$, $p=2$ and $R_c=Q_c$ ($c\in\Sigma$).

To define $F$ on $X^{(1)}$
we fix a number $\omega\in (0,1)\setminus Q$ and put
$$
 c_{r}=a_{r \omega}
 \qquad
 (r\in Q).
$$
Take any $r\in Q$ and put $J_r=[\varphi_\emptyword(r),\tau_\emptyword(r)]$, $s=\tau_\emptyword(r)$.
Let $d\in\Sigma$ be such that $r\in Q_d$.
Two applications of Lemma~\ref{L:tau-12}, first with
$I=[0,\omega]$, $I'=J_r$, $R_c=Q_c\cap (0,\omega)$, $R'_c=Q_{\last{d}}\cap (\varphi_\emptyword(r), s)$ ($c\in\Sigma$),
then with $I=[\omega,1]$, $I'=[0,1]$, $R_c=Q_c\cap (\omega,1)$, $R'_c=Q_{\last{dc}}$ ($c\in\Sigma$),
and arbitrary $L>1/\min\{\omega,1-\omega\}$,
give maps $\tau_r':[0,\omega]\to J_r$ and $\tau_r'':[\omega,1]\to [0,1]$.
We define $F$ on $A_r$ by
\begin{equation}\label{EQ:def-F-X1}
 F(a_{rt})=
 \begin{cases}
   a_{\tau_r'(t)} &\text{if } t\in [0,\omega];
 \\
   a_{s \tau_r''(t)} &\text{if } t\in [\omega,1].
 \end{cases}
\end{equation}
Since $\tau_r'(\omega)=s$, $\tau_r''(\omega)=0$ and $a_{s0}=a_s$, the map $F$ is well defined
on $A_r$. To shorten the notation, define a (not continuous) map $\tau_r:[0,1]\to[0,1]$
by $\tau_r(t)=\tau_r'(t)$ for $t\le \omega$ and $\tau_r(t)=\tau_r''(t)$ for $t> \omega$.

Now assume that $m\ge 2$ and $F$ is defined on $X^{(m-1)}$. Fix arbitrary
$\alpha=\alpha_0\dots\alpha_{m-1}\in Q^m$
and let $\gamma\in \Sigma^m$ be such that $\alpha\in Q_\gamma$.
Since $a_\alpha\in X^{(m-1)}$, $F(a_\alpha)$ has already been defined and
there is $\beta\in Q^*$ such that $F(a_\alpha)=a_\beta$.
If $\alpha_1<\omega$, take a map $\tau_\alpha:[0,1]\to[0,1]$
obtained from Lemma~\ref{L:tau-12} for $R_c=Q_c$, $R'_c=Q_{\last\gamma}$ ($c\in\Sigma$) and $L=L_m$.
Otherwise ($\alpha_1>\omega$) let $\tau_\alpha:[0,1]\to[0,1]$
be a map obtained from Lemma~\ref{L:tau-12} for $R_c=Q_c$, $R'_c=Q_{\last{\gamma c}}$
($c\in \Sigma$) and $L=L_m$.
Now we can define $F$ on $A_\alpha$ by
\begin{equation}\label{EQ:def-F-X2}
 F(a_{\alpha t})=
   a_{\beta \tau_\alpha(t)}
 \qquad\text{for every } t\in [0,1].
\end{equation}

In this way we have defined $F$ on $\bigcup_m X^{(m)}$. Before giving the definition of $F$ on
$X^{(\infty)}$ we summarize the crucial properties of the index maps $\tau_\alpha$, which
immediately follows from Lemmas~\ref{L:tau-0} and \ref{L:tau-12}.

\begin{lemma}\label{L:tau-alpha-props}
 Let $\gamma\in\Sigma^*$, $\alpha\in Q_\gamma$ and $c\in \Sigma$. Then
 \begin{enumerate}
   \item[(a)] $\tau_\alpha(Q_c)\subseteq Q_{\last\gamma}$,
      with equality if $\abs{\alpha}\ge 2$ and $\alpha_1<\omega$;
   \item[(b)] $\tau_\alpha(Q)=Q$ and $\tau_\alpha(Q_c)=Q_{\last{\gamma c}}$
      provided $\alpha=\emptyword$, or  $\abs{\alpha}\ge 2$ and $\alpha_1>\omega$;
   \item[(c1)] $\tau'_\alpha(Q_c\cap(0,\omega))
     = Q_{\last\gamma}\cap (\varphi_\emptyword(\alpha), \tau_\emptyword(\alpha))
     \subseteq Q_{\last\gamma}$ provided $\abs{\alpha}=1$.
   \item[(c2)]
    $\tau''_\alpha(Q\cap(\omega,1))=Q$ and
    $\tau''_\alpha(Q_c\cap(\omega,1))=Q_{\last{\gamma c}}$ provided $\abs{\alpha}=1$.
 \end{enumerate}
\end{lemma}

By the construction, for every $\alpha\in Q^*$
there is (unique) $\varrho(\alpha)\in Q^*$ with
\begin{equation}\label{EQ:rho-def}
  F(b_\alpha)=b_{\varrho(\alpha)}.
\end{equation}
Moreover, (\ref{EQ:def-F-X0})--(\ref{EQ:def-F-X2}) imply that, for every $r,s\in Q$
and $\alpha\in Q^*$ with $\abs{\alpha}\ge 2$,
\begin{equation}\label{EQ:rho-props}
\begin{split}
  &\varrho(\emptyword)=\emptyword,\quad
  \varrho(r)=\tau_\emptyword(r),\quad
  \varrho(rs)=\begin{cases}
    \tau_r'(s)             &\text{if } s<\omega,
   \\
    \varrho(r)\tau_r''(s)  &\text{if } s>\omega
  \end{cases}
\\
  &\quad\text{and}\quad
  \varrho(\alpha s)=\varrho(\alpha)\tau_\alpha(s).
\end{split}
\end{equation}

\begin{lemma}\label{L:F-props-varrho} Let $\alpha\in Q^+$, $m=\abs{\alpha}$
and $\beta=\varrho(\alpha)$. Put $\theta_\alpha=1$ if $m\ge 2$ and $\alpha_1<\omega$, and
$\theta_\alpha=0$ otherwise.
Then
\begin{enumerate}
  \item[(a)] $\abs{\beta}=m-\theta_\alpha$;
  \item[(b)] $\beta_{i-\theta_\alpha} = \tau_{\alpha_0\dots\alpha_{i-1}}(\alpha_i)$
    for every $\theta_\alpha\le i<m$;
  \item[(c)] $\varrho(\alpha\alpha')$ starts with $\beta$ for every $\alpha'\in Q^*$;
  \item[(d)] the sequence $(\abs{\varrho^n(\alpha)})_{n\ge 0}$ is (not necessarily strictly) decreasing
    and is eventually $1$.
\end{enumerate}
\end{lemma}
\begin{proof}
The statements (a)--(c) immediately follow from (\ref{EQ:rho-props}).
The first part of (d) and the fact that $\abs{\varrho^n(\alpha)}\ge 1$ for every $n$ follow from (a).
To finish the proof it suffices to show that for every $\alpha$ with $\abs{\alpha}\ge 2$
there is $n\in\NNN$ such that $\abs{\varrho^n(\alpha)}<\abs{\alpha}$.
Suppose, to the contrary, that there is $\alpha$ with $m=\abs{\alpha}\ge 2$ such that
$\abs{\varrho^n(\alpha)}=m$ for every $n$. Put $\beta^n=\beta_0^n\dots\beta_{m-1}^n=\varrho^n(\alpha)$
($n\ge 0$).
Then $\beta_1^n>\omega$ and, by (\ref{EQ:rho-props}),  $\beta_1^{n+1}=\tau''_{\beta_0^n}(\beta_1^n)$
for every $n$.
Notice that, by Lemma~\ref{L:tau-12}(e),
$\tau''_r(t)\le (t-\omega)/(1-\omega)<t$ for every $t\in (0,1)$ and $r\in Q$.
Hence
\begin{equation*}
  \alpha_1-\beta_1^n
  > \frac{\alpha_1-\omega}{1-\omega}  -   \frac{\beta_1^{n-1}-\omega}{1-\omega}
  = \frac{\alpha_1-\beta_1^{n-1}}{1-\omega}
\end{equation*}
for every $n\ge 1$. Thus $\alpha_1-\beta_1^n > (1-\omega)^{-(n-1)} (\alpha_1-\beta_1^1)\to\infty$
for $n\to\infty$ (note that $\beta_1^1=\tau''_{\alpha_0}(\alpha_1)<\alpha_1$).
But this contradicts the fact that $\beta_1^n>\omega$ for every $n$.
\end{proof}

The next lemma, which trivially follows from Lemma~\ref{L:F-props-varrho}, enables us
to define $F$ on $X^{(\infty)}$.
\begin{lemma}\label{L:F-props-varrho-infty} Let $\alpha\in Q^\infty$ and
put $\theta_\alpha=1$ if $\alpha_1<\omega$ and
$\theta_\alpha=0$ if $\alpha_1>\omega$. Then there is exactly
one $\beta\in Q^*$, which will be denoted by $\varrho(\alpha)$, such that
\begin{equation*}
 \varrho(\alpha_0\dots \alpha_{m-1}) = \beta_0\dots\beta_{m-1-\theta_\alpha}
 \qquad\text{for every } m\ge 2.
\end{equation*}
Moreover, $\beta_{i-\theta_\alpha} = \tau_{\alpha_0\dots\alpha_{i-1}}(\alpha_i)$
for every $i\ge\theta_\alpha$.
\end{lemma}

By Lemma~\ref{L:F-props-varrho-infty} we can define $F$ on $X^{(\infty)}$ by
\begin{equation}\label{EQ:def-F-Xinfty}
 F(b_{\alpha})=
   b_{\varrho(\alpha)}
 \qquad\text{for every } \alpha\in Q^\infty.
\end{equation}

\section{Properties of the map $F$}\label{S:F-props}
In this section we deal with the dynamical properties of the map $F$ constructed
in the previous section.

\subsection{Basic properties of $F$}\label{SS:F-props}
\begin{lemma}\label{L:F-props-Xm-alpha} Let $\alpha\in Q^*$ have $\abs{\alpha}\ge 2$.
Put $\beta=\varrho(\alpha)$. Then
\begin{enumerate}
  \item[(a)] $F(a_\alpha)=a_\beta$, $F(b_\alpha)=b_\beta$ and
    $F|_{A_\alpha}:A_\alpha\to A_\beta$ is a continuous surjection;
  \item[(b)] $F(X_\alpha)\subseteq X_\beta$.
\end{enumerate}
\end{lemma}
\begin{proof}
 The statement (a) follows from (\ref{EQ:def-F-X1}), (\ref{EQ:def-F-X2}) and
 the choice of $\tau_\alpha$. To prove (b) realize that $\bigcup_{\alpha'\in Q^*} A_{\alpha\alpha'}$
 is dense in $X_\alpha$ and that $F(A_{\alpha \alpha'})=A_{\varrho(\alpha \alpha')}\subseteq X_\beta$
 by (\ref{EQ:def-F-X2}). Thus $F(X_\alpha)\subseteq \closure{X_\beta}=X_\beta$.
\end{proof}

\begin{lemma}\label{L:rho-surj}
The maps $\varrho|_{Q^*}:Q^*\to Q^*$ and $\varrho|_{Q^\infty}:Q^\infty\to Q^\infty$
are surjections.
Moreover, for every $\beta\in Q^*$ ($\beta\in Q^\infty$)
there is $\alpha\in Q^*$ ($\alpha\in Q^\infty$) such that
\begin{equation}\label{EQ:L:rho-surj}
  \varrho(\alpha)=\beta \qquad
  \text{and} \qquad
  \abs{\alpha}=\abs{\beta}.
\end{equation}
\end{lemma}
Note that if $\abs{\beta}\ge 2$ and $\alpha$ satisfies (\ref{EQ:L:rho-surj}) then $\alpha_1>\omega$.
\begin{proof}
First we deal with $\varrho$ on $Q^*$.
Take any $\beta\in Q^*$.
If $\beta=\emptyword$ put $\alpha=\emptyword$.
If $\abs{\beta}=1$, (\ref{EQ:L:rho-surj}) is guaranteed
by the fact that $\tau_\emptyword(Q)=Q$, see Lemma~\ref{L:tau-alpha-props}(b).
If $\abs{\beta}=2$, (\ref{EQ:L:rho-surj})
follows from  Lemma~\ref{L:tau-alpha-props}(c2).
Finally, for $\abs{\beta}\ge 3$ the result follows by induction from (\ref{EQ:rho-props}) and Lemma~\ref{L:tau-alpha-props}(b).

Now take any $\beta\in Q^\infty$.
By the previous case there are  $\alpha_0,\alpha_1\in Q$ such that
$\varrho(\alpha_0\alpha_1)=\beta_0\beta_1$ and $\alpha_1>\omega$.
Define inductively $\alpha_2,\alpha_3,\dots$ in such a way that
$\tau_{\alpha_0\dots\alpha_{i-1}}(\alpha_i)=\beta_i$ for every $i\ge 2$; this is possible
by Lemma~\ref{L:tau-alpha-props}(b). Then, by Lemma~\ref{L:F-props-varrho-infty},
$\varrho(\alpha_0\alpha_1\dots)=\beta$.
\end{proof}

\begin{proposition}\label{P:F-cont-surj}
The constructed map $F:X\to X$ is a non-injective continuous surjection.
Moreover, the sets $X^{(\infty)}$ and $X^{(m)}$ ($m\ge 0$) are (strongly) $F$-invariant.
\end{proposition}
\begin{proof}
Surjectivity of $F$ follows from Lemmas~\ref{L:rho-surj}, \ref{L:F-props-Xm-alpha}(a)
and the facts that $F(A_\emptyword)=A_\emptyword$, $F(A_r)\supseteq A_{\varrho(r)}$ for every $r\in Q$.
Further, $F$ is non-injective since, for example, every point of $(\varphi_\emptyword(r),\tau_\emptyword(r)]$
($r\in Q$) has at least two preimages: one in $A_\emptyword$ and another one in $A_r$.

To prove continuity of $F$ take any convergent sequence $(x_n)_n$ of points from $X$
such that, for some $y\in X$, $y=\lim_n F(x_n)$; put $x=\lim_n x_n$. We want to show that $y=F(x)$.
First assume that $x\in X^{(\infty)}$, that is, $x=b_\alpha$ for some $\alpha\in X^\infty$.
Take any $m\ge 2$ and put $\alpha'=\alpha_0\dots\alpha_{m-1}$.
Then there is $n_0$ such that
$x_n\in X_{\alpha'}$ for every $n\ge n_0$. Hence, by Lemma~\ref{L:F-props-Xm-alpha}(b),
$F(x_n)\in X_{\varrho(\alpha')}$ for every $n\ge n_0$ and so $y\in X_{\varrho(\alpha')}$.
Since this is true for every $m$ we have $y=b_{\varrho(\alpha)}=F(x)$.

Now assume that $x\in X^{(m)}$ for some $m$; that is, either $x=a_\emptyword$ or
$x=a_{\alpha t}$ for some $\alpha\in\bigcup_{0\le k\le m} Q^k$ and $t\in (0,1]$.
We describe only the case when $x=a_{\alpha t}$ with $\abs{\alpha}\ge 2$; the other cases
can be described analogously.
Put $\beta=\varrho(\alpha)$ and $s=\tau_\alpha(t)$.
Lemma~\ref{L:F-props-Xm-alpha} and (\ref{EQ:rho-props}) imply that,
for every sufficiently large $n$, $x_n\in X_{\alpha r_n}$ and $F(x_n)\in X_{\beta s_n}$
for some $r_n\in Q$ such that $\lim_n r_n= t$ and $\lim_n s_n=s$, where
$s_n=\tau_\alpha(r_n)$.
If there are infinitely many $n$'s with $r_n\ne t$, then $\diam X_{\alpha r_n}\to 0$;
moreover, $\diam X_{\beta s_n}\to 0$. (In fact, either $s\not\in Q$ and this is obvious,
or $s\in Q$ and $\tau_\alpha^{-1}(s)$ is finite by Lemma~\ref{L:tau-12}(d),
so $s_n\ne s$ for all but finitely many $n$s.) Thus, by (\ref{EQ:def-F-X2}),
$y=a_{\beta s}=a_{\varrho(\alpha)\tau_\alpha(t)} = F(a_{\alpha t})=F(x)$.
If there are only finitely many $n$'s with $r_n\ne t$, we may assume that every $r_n=t$.
Then there is a sequence $(r_n')_n$ in $Q$ converging to $0$
such that $x_n\in X_{\alpha t r_n'}$ for every $n$.
Using that $\tau_\alpha^{-1}(0)=\{0\}$ by Lemma~\ref{L:tau-12}(b),
analogously as before
we obtain that $y=F(x)$.

Finally, the fact that the sets $X^{(m)}$ ($m\ge 0$) and $X^{(\infty)}$ are strongly $F$-invariant
is a consequence of the construction and surjectivity of $F$.
\end{proof}

\subsection{Regular periodic decompositions for $F$}\label{SS:F-RPD}
The main result of this section is Proposition~\ref{P:F-RPD} stating that every $\DDd_m$
is a regular periodic decomposition for $F$. We start with the next lemma.

\begin{lemma}\label{L:F-RPD-rho}
Let $\alpha\in Q^*$ and $\gamma\in\Sigma^*$ be such that $\alpha\in Q_\gamma$. Then
\begin{enumerate}
  \item[(a)] $\varrho(\alpha)\in Q_{\tilde\gamma+1}$;
  \item[(b)] $\varrho(\alpha)\in Q_{\gamma+1}$ provided $\abs{\alpha}\le 1$,
    or $\abs{\alpha}\ge 2$ and $\alpha_1>\omega$.
\end{enumerate}
\end{lemma}
\begin{proof}
 (a) If $\alpha=\emptyword$ then (a) is trivial (use that $Q_\emptyword=\{\emptyword\}$
 and $\emptyword+1=\emptyword$). Assume that $\varrho(\alpha)\in Q_{\tilde\gamma+1}$
 for some $\alpha\in Q_\gamma$ and take any $c\in\Sigma$, $r\in Q_c$.
 Then $\varrho(\alpha r) = \varrho(\alpha)\tau_\alpha(r)\in Q_{(\tilde\gamma+1) \last\gamma}
 =Q_{\gamma+1}$ by Lemma~\ref{L:tau-alpha-props}(a) and (\ref{EQ:e-gamma}). Hence, by induction,
 (a) is proved.

 (b) The assertion is again trivially true for $\alpha=\emptyword$. For $\alpha\in Q^1$,
 Lemma~\ref{L:tau-alpha-props}(b) gives
 $\varrho(\alpha)=\tau_\emptyword(\alpha)\in Q_{\last{\emptyword \gamma}}=Q_{\gamma+1}$.
 Further, if $c\in\Sigma$ and $r\in Q_c\cap(\omega,1)$, then
 (\ref{EQ:rho-props}) and Lemma~\ref{L:tau-alpha-props}(c2) yield
 $\varrho(\alpha r) = \varrho(\alpha) \tau''_\alpha(r)\in Q_{(\gamma+1) \last{\gamma c}}
 =Q_{\gamma c+1}$. So (b) is proved for every $\alpha$ with $\abs{\alpha}\le 1$
 or with $\abs{\alpha}=2$, $\alpha_1>\omega$.

 To finish the proof of (b), it suffices to show that
 $\varrho(\alpha r)\in Q_{\gamma c+1}$ for
 any $\alpha$ with $\abs{\alpha}\ge 2$, $\alpha_1>\omega$ and $r\in Q_c$, $c\in\Sigma$.
 But this immediately follows by induction from (\ref{EQ:rho-props}) and Lemma~\ref{L:tau-alpha-props}(b), since
 $\varrho(\alpha r) =  \varrho(\alpha) \tau_\alpha(r)\in Q_{(\gamma+1) \last{\gamma c}}
 =Q_{\gamma c+1}$. So the proof is finished.
\end{proof}

To prove Proposition~\ref{P:F-RPD}, we will also need the following refinement of
Lemma~\ref{L:F-props-Xm-alpha}(b). Let us note that in Lemma~\ref{L:F-props-image-of-Xalpha}
even the equalities hold; see Lemma~\ref{L:F-props-image-of-Xalpha-equalities}.

\begin{lemma}\label{L:F-props-image-of-Xalpha}
Let $\alpha\in Q^+$, $\beta=\varrho(\alpha)$ and $\gamma$ be such that $\alpha\in Q_\gamma$. Then
$$
 F(X_\alpha)\subseteq
 \begin{cases}
   X_\beta
   &\text{if }\abs\alpha\ge 2 \text{ and } \alpha_1>\omega;
 \\
   A_\beta \cup \left(\bigsqcup_{s\in Q_{\last\gamma}} X_{\beta s}\right)
   &\text{if }\abs\alpha\ge 2 \text{ and } \alpha_1<\omega;
 \\
   [F(a_{\alpha}), a_\beta] \cup \left(\bigsqcup_{s\in Q_{\last\gamma}\cap(\varphi_\emptyword(\alpha),\beta]} X_{s}\right)
   &\text{if }\abs\alpha=1.
 \end{cases}
$$
\end{lemma}
\begin{proof}
The first case was proved in Lemma~\ref{L:F-props-Xm-alpha}(b). Assume now that $\abs{\alpha}\ge 2$
and $\alpha_1<\omega$. Write $X_\alpha$ in the form
\begin{equation}\label{EQ:Xalpha=cup-Xalphar}
  X_\alpha = A_\alpha \cup\left(
    \bigsqcup_{r\in Q} X_{\alpha r}
  \right).
\end{equation}
By Lemma~\ref{L:F-props-Xm-alpha} and (\ref{EQ:rho-props}), $F(A_\alpha)=A_\beta$
and $F(X_{\alpha r})\subseteq X_{\varrho(\alpha r)}=X_{\beta \tau_\alpha(r)}$
for every $r\in Q$.
Since $\tau_\alpha(r)\in Q_{\last\gamma}$ by Lemma~\ref{L:tau-alpha-props}(a), the second inclusion is
proved.

Finally assume that $\abs{\alpha}=1$. Again write $X_\alpha$ in the form (\ref{EQ:Xalpha=cup-Xalphar})
and realize that $F(A_\alpha)=A_\beta\cup [F(a_\alpha),a_\beta]$ and $F(X_{\alpha r})\subseteq
X_{\varrho(\alpha r)}$. If $r<\omega$ then $\varrho(\alpha r) = \tau'_\alpha (r)\in Q_{\last\gamma}
\cap(\varphi_\emptyword(\alpha),\beta)$
by Lemma~\ref{L:tau-alpha-props}(c1); thus
$F(X_{\alpha r})\subseteq X_s$ for some $s\in Q_{\last\gamma}\cap(\varphi_\emptyword(\alpha),\beta)$.
On the other hand, if $r>\omega$ then $\varrho(\alpha r) = \beta\tau''_\alpha (r)$
and so $F(X_{\alpha r})\subseteq X_\beta$. Hence also the final inclusion is proved,
since $\beta\in Q_{\gamma+1}=Q_{\last\gamma}$ by Lemma~\ref{L:F-RPD-rho}(b).
\end{proof}

Now we are able to show that every $\DDd_m$ is a regular periodic decomposition for $F$; recall the definitions (\ref{EQ:RPD-Dgamma-def}) and (\ref{EQ:RPD-Dm-def}) of $D_\gamma$ and $\DDd_m$.

\begin{proposition}\label{P:F-RPD}
 For every $m\ge 0$, $\DDd_m$ is a regular periodic decomposition for $F$.
\end{proposition}
\begin{proof}
For $m=0$ this is trivial, so assume that $m\ge 1$. We will prove that
\begin{equation}\label{EQ:F-RPD-proof}
 F(X_\alpha)\subseteq D_{\gamma+1}
 \qquad\text{for every } \gamma\in\Sigma^m \text{ and } \alpha\in Q_\gamma;
\end{equation}
from this the proposition will follow since (\ref{EQ:F-RPD-proof})
implies that $F(\interior{D_\gamma}) \subseteq D_{\gamma+1}$, see Lemma~\ref{L:def-RPD}(a).

For $m=1$ Lemma~\ref{L:F-props-image-of-Xalpha} implies
\begin{equation*}
 F(X_\alpha)
 \subseteq A_\emptyword \cup\left(
   \bigsqcup_{s\in Q_{\gamma+1}} X_{s}
 \right)
 = D_{\gamma+1}.
\end{equation*}
Assume now that $m\ge 2$. Then, by Lemma~\ref{L:F-props-Xm-alpha}(b),
$F(X_\alpha)\subseteq X_{\varrho(\alpha)}$.
If $\alpha_1>\omega$ then $\varrho(\alpha)\in Q_{\gamma+1}$ by Lemma~\ref{L:F-RPD-rho},
hence (\ref{EQ:F-RPD-proof}) is true.
On the other hand, for $\alpha_1<\omega$, Lemma~\ref{L:F-props-image-of-Xalpha}
and the facts that $\beta\in Q_{\tilde{\gamma}+1}$ and
$\beta s\in Q_{\gamma+1}$ for every $s\in Q_{\last\gamma}$ give
\begin{equation*}
 F(X_\alpha)
 \subseteq A_\beta \cup\left(
   \bigsqcup_{s\in Q_{\last\gamma}} X_{\beta s}
 \right)
 \subseteq Y^{(\tilde\gamma+1)} \cup\left(
   \bigsqcup_{\beta'\in Q_{\gamma+1}} X_{\beta'}
 \right)
 = D_{\gamma+1}
\end{equation*}
(realize that $(\gamma+1)^{\widetilde{ }} = \tilde\gamma+1$).
The proposition is proved.
\end{proof}

\subsection{Transitivity of $F$}\label{SS:F-trans}
Here we show that $F$ is transitive, see Proposition~\ref{P:F-trans}.
As before, we start with some technical lemmas.
For every $\alpha\in Q^+$ put
\begin{equation}\label{EQ:F-trans-Ralpha-def}
  R_\alpha = \{
   \beta\in Q^+:\ \varrho^n(\alpha\alpha') \text{ starts with } \beta
   \text{ for some } \alpha'\in Q^* \text{ and } n\in\NNN
  \}.
\end{equation}
Note that, in (\ref{EQ:F-trans-Ralpha-def}), one can replace $\alpha'\in Q^*$ by $\alpha'\in Q^\infty$.

\begin{lemma}\label{L:F-trans-Ralpha-implies-trans}
 The constructed map $F:X\to X$ is transitive if and only if $R_\alpha=Q^+$ for every $\alpha\in Q^+$.
\end{lemma}
\begin{proof}
Assume that $F$ is transitive. Take any $\alpha,\beta\in Q^+$ and put $U=X_\alpha\setminus\{a_\alpha\}$,
$V=X_\beta\setminus\{a_\beta\}$. Transitivity implies that
there is $x\in U$ and $n\in\NNN$ such that $F^n(x)\in V$. The continuity of $F$ at $x$ implies that
$F^n(U_x)\subseteq V$ for some neighborhood $U_x\subseteq U$ of $x$. Take any $x'=b_{\alpha\alpha'}
\in U_x\cap X^{(\infty)}$. Then, by (\ref{EQ:def-F-Xinfty}), $F^n(x')=b_{\varrho^n(\alpha\alpha')}
\in X_\beta$. Hence $\beta\in R_\alpha$. Since $\beta$ was arbitrary, we have that $R_\alpha=Q^+$.

Assume now that $R_\alpha=Q^+$ for every $\alpha\ne\emptyword$.
Take any non-empty open sets $U,V$. Then there are $\alpha,\beta\in Q^+$ such that
$X_\alpha\subseteq U$ and $X_\beta\subseteq V$; we may assume that
$\abs{\alpha}\ge 2$. Since $\beta\in R_\alpha$, there is $\alpha'\in Q^\infty$
and $n\in\NNN$ such that $\varrho^n(\alpha\alpha')$ starts with $\beta$.
But then $F^n(b_{\alpha\alpha'})=b_{\varrho^n(\alpha\alpha')}$ belongs to $X_\beta$,
so $F^n(U)\cap V \supseteq F^n(X_\alpha)\cap X_\beta\ne\emptyset$. This proves transitivity of $F$.
\end{proof}

Thus, by the previous lemma, to show the transitivity of $F$
we need to study properties of the iterates of $\varrho$.
For $\alpha\in Q^*$, let $[\alpha]$ denote the cylinder of all words from $Q^*$ starting with $\alpha$; that is,
\begin{equation}\label{EQ:alpha-cylinder}
  [\alpha] = \{\alpha\beta:  \beta\in Q^*\}.
\end{equation}
Note that if $\alpha\in Q_\gamma$ and $\delta\in\Sigma^*$, then
\begin{equation*}
    [\alpha]\cap Q_{\gamma\delta}= \{\alpha\beta: \ \beta\in Q_\delta\}.
\end{equation*}
For $k\ge 1$ let $\pi_k:\left(\bigcup_{m\ge k} \Sigma^m\right) \to \Sigma^k$
be the projection onto the first $k$ letters, that is,
$\pi_k(\gamma)=\gamma_0\dots\gamma_{k-1}$ for every $\gamma=\gamma_0\dots\gamma_{m-1}\in \Sigma^m$ with $m\ge k$.

\begin{lemma}\label{L:F-trans-rho}
For every $\alpha\in Q_\gamma$ with $\abs{\alpha}\ge 2$, $k\ge 1$
and $\delta\in\Sigma^k$,
\begin{equation*}
  \varrho([\alpha]\cap Q_{\gamma\delta})
  = [\varrho(\alpha)] \cap Q_{\pi_{p+k}(\gamma\delta+1)}
  = \begin{cases}
    [\varrho(\alpha)] \cap Q_{\gamma\tilde{\delta}+1}  &\text{if } \alpha_1<\omega;
  \\
    [\varrho(\alpha)] \cap Q_{\gamma\delta+1}  &\text{if } \alpha_1>\omega,
  \end{cases}
\end{equation*}
where $p=\abs{\varrho(\alpha)}$.
\end{lemma}
\begin{proof}
We give the proof only for the case when $\alpha_1<\omega$; the other case can be described analogously.
First take any $d\in\Sigma$. Then, by (\ref{EQ:rho-props}),
$$
 \varrho([\alpha]\cap Q_{\gamma d})
 = \{\varrho(\alpha r): r\in Q_d\}
 = \{\varrho(\alpha) \tau_\alpha(r): r\in Q_d\}
 = \{\varrho(\alpha) s: s\in \tau_\alpha(Q_d)\}.
$$
Since $\tau_\alpha(Q_d)=Q_{\last{\gamma}}$
by Lemma~\ref{L:tau-alpha-props}(a),
we have
\begin{equation}\label{EQ:F-trans-rho-1}
  \varrho([\alpha]\cap Q_{\gamma d}) =
  [\varrho(\alpha)] \cap Q_{\gamma+1}.
\end{equation}

Now we continue by induction on $k=\abs{\delta}$. We have proved the lemma for $k=1$.
Assume now that $k\ge 1$ and that the lemma is true for every $\delta$ with $\abs{\delta}\le k$.
Take arbitrary $\delta\in \Sigma^+$ and $d\in \Sigma$. Then
\begin{equation*}
  \varrho([\alpha]\cap Q_{\gamma\delta d})
  = \bigcup_{\beta\in Q_\delta}  \varrho([\alpha\beta] \cap Q_{\gamma\delta d}).
\end{equation*}
So, by (\ref{EQ:F-trans-rho-1}),
\begin{equation*}
  \varrho([\alpha]\cap Q_{\gamma\delta d})
  = \bigcup_{\beta\in Q_\delta}  \left\{ \varrho([\alpha\beta]) \cap Q_{\gamma\delta+1} \right\}
  = \varrho([\alpha]\cap Q_{\gamma\delta})
    ~\cap ~ Q_{\gamma\delta+1} .
\end{equation*}
By the induction hypothesis, $\varrho([\alpha]\cap Q_{\gamma\delta})=[\varrho(\alpha)]\cap Q_{\gamma\tilde{\delta}+1}$.
Since $(\delta d)^{\widetilde{ }}=\delta$
and $Q_{\gamma\tilde{\delta}+1}\supseteq Q_{\gamma{\delta}+1}$,
the lemma is proved for $\delta'=\delta d$. Thus, by induction, the lemma is proved for every
$\delta\in\Sigma^+$.
\end{proof}

\begin{lemma}\label{L:F-trans-rho-itetate}
Let $\alpha\in Q_\gamma$ and $n\ge 1$ be such that $\abs{\varrho^{n-1}(\alpha)}\ge 2$;
put $p=\abs{\varrho^n(\alpha)}$.
Then, for every $k\ge 1$ and $\delta\in\Sigma^k$,
\begin{equation*}
  \varrho^n([\alpha]\cap Q_{\gamma\delta})
  = [\varrho^n(\alpha)] \cap Q_{\pi_{p+k}(\gamma\delta+n)}.
\end{equation*}
Consequently, for every $k\ge 1$,
\begin{equation*}
  \varrho^n\left([\alpha]\cap (Q_{\gamma}\times Q^k) \right)
  = [\varrho^n(\alpha)]
    \cap
    \left(  Q_{\gamma+n}\times Q^{p+k-\abs{\alpha}} \right).
\end{equation*}
\end{lemma}
Note that $(\abs{\alpha}-\abs{\varrho^n(\alpha)})$ is the number of integers $0\le i<n$ for which $(\varrho^i(\alpha))_1<\omega$.
\begin{proof}
The first statement follows by induction from Lemma~\ref{L:F-trans-rho}
(one needs to use that $\pi_l(\nu+m)=\pi_l(\nu)+m$ and $\pi_{l'}\circ\pi_l=\pi_{l'}$ for every
$m,\nu$ and $l'<l\le\abs\nu$).
To prove the second one,
it suffices to realize that $\bigcup_{\delta\in\Sigma^k} Q_{\gamma\delta}=Q_\gamma\times Q^k$
and $\bigcup_{\delta\in\Sigma^k} Q_{\pi_{p+k}(\gamma\delta+n)}=Q_{\gamma+n}\times Q^{p+k-\abs{\gamma}}$.
\end{proof}

\begin{lemma}\label{L:F-trans-rho-itetate1}
 For every $s\in Q_c$ ($c\in\Sigma$), $\delta\in \Sigma^k$ ($k\ge 1$) and $n\ge 1$,
\begin{equation*}
  \varrho^n\left([s]\cap Q_{c \delta} \right)
  \supseteq
  [\varrho^n(s)]  \cap  Q_{c\delta+n}
\quad\text{and}\quad
  \varrho^n\left([s]\cap (Q_{c} \times Q^k) \right)
  \supseteq
  [\varrho^n(s)]  \cap  \left(Q_{c+n} \times Q^k \right).
\end{equation*}
\end{lemma}
\begin{proof}
It suffices to prove the first inclusion. For $n=1$ this immediately follows from Lemma~\ref{L:tau-alpha-props}(c2)
and (b).
(In fact, $\varrho([s]\cap Q_{c\delta})\supseteq\{\varrho(s\alpha): \alpha\in Q_\delta, \alpha_0>\omega\}$
and, for $\alpha=\alpha_0\dots\alpha_{k-1}$ with $\alpha_0>\omega$, $\varrho(s\alpha)=\tau_\emptyword(s)\tau_s''(\alpha_0)
\tau_{s\alpha_0}(\alpha_1)\dots\tau_{s\alpha_0\dots\alpha_{k-2}}(\alpha_{k-1})$ by Lemma~\ref{L:F-props-varrho}(b).
So, by Lemma~\ref{L:tau-alpha-props}(c2) and (b),
$\varrho([s]\cap Q_{c\delta})\supseteq\tau_\emptyword(s)\times Q_{\last{c\delta_0}}\times Q_{\last{c\delta_0\delta_1}}
\times\dots\times Q_{\last{c\delta_0\dots\delta_{k-1}}} = [\tau_\emptyword(s)]\cap Q_{c\delta+1}$.)
The proof for general $n$ is by simple induction.
\end{proof}

\begin{lemma}\label{L:F-trans-Ralpha-is-Q+0}
 For every $\alpha\in Q_\gamma$ with $k=\abs{\alpha}\ge 1$, every $s\in Q_d$ ($d\in\Sigma$) and every
 $\delta\in\Sigma^{k}$ there are $n\ge 1$, $c\in\Sigma$ and $u\in Q_c$ with
 \begin{equation*}
    \varrho^n(\alpha u) = s
    \quad\text{and}\quad
    \gamma c+n = d\delta.
 \end{equation*}
 Consequently, $R_\alpha\supseteq Q$.
\end{lemma}
\begin{proof}
Fix arbitrary $d\in \Sigma$ and $s\in Q_d$.
We first prove the lemma for $k=1$. To this end, take any $\alpha\in Q_\gamma$ with $\abs{\alpha}=1$
and any $\delta\in\Sigma$.
Recall that $\varrho(t)=\tau_\emptyword(t)$ for every $t\in Q$ and $\tau_\emptyword$ was given by Lemma~\ref{L:tau-0}.
By Lemma~\ref{L:tau-0}(g) there are $h,k\in\NNN_0$ and
$t\in Q_{\gamma+(k+1)}\cap (\varphi_\emptyword(r'),\tau_\emptyword(r'))$
(where $r'=\varrho^k(\alpha)=\tau_\emptyword^k(\alpha)$) such that
$\varrho^h(t)=s$.
Put $n=k+1+h$ and take $c\in\Sigma$ such that $\gamma c + n=d\delta$.
By Lemma~\ref{L:tau-alpha-props}(c1) there is
$u'\in Q_{\last{\gamma c+k}}$, $u'<\omega$ with $\tau'_{r'}(u')=t$; thus $\varrho(r'u')=t$.
Repeated application of Lemma~\ref{L:tau-alpha-props}(c2) gives the existence
of $u\in Q_c$ with $\varrho^k(\alpha u)=r'u'$.
So $\varrho^{n}(\alpha u)=s$ and, by the choice of $c$, $\gamma c+n = d\delta$. Thus the case $k=1$ is proved.

Assume now that $k\ge 1$ and the lemma is true for every $\alpha$ with $\abs{\alpha}\le k$.
Take any $\alpha'=\alpha r\in Q_{\gamma e}$ with $\abs{\alpha}=k$, $r\in Q_e$ ($\abs{r}=1$), and any
$\delta'=\delta f$ with $\abs{\delta}=k$ and $f\in\Sigma$.
Let $p\ge 0$ be such that $k+1=\abs{\alpha r} = \abs{\varrho^p(\alpha r)}>\abs{\varrho^{p+1}(\alpha r)}=k$.
Put $\bar\alpha=\varrho^{p+1}(\alpha r)$.
Since $\abs{\bar\alpha}=k$ and $\bar\alpha\in Q_{\bar\gamma}$ ($\bar{\gamma}=\tilde\gamma+p+1$),
by the induction hypothesis there are $\bar{n}\ge 1$, $\bar{c}\in\Sigma$ and
$\bar{u}\in Q_{\bar c}$ with
\begin{equation*}
    \varrho^{\bar n}(\bar \alpha \bar u) = s
    \qquad
    \text{and}
    \qquad
    \bar\gamma \bar c + \bar n = d \delta.
\end{equation*}
By Lemma~\ref{L:F-trans-rho-itetate},
$\bar\alpha \bar u\in \varrho^{p+1}\left([\alpha r]\cap (Q_{\gamma e}\times Q) \right)
=[\bar\alpha] \cap Q_{\gamma e + (p+1)}$.
Thus, by Lemma~\ref{L:F-RPD-rho},
$\gamma e + (p+1)=\bar\gamma \bar c$,
and so $\gamma e + n=d\delta$, where $n=p+1+\bar n$.
Hence there is $c\in\Sigma$ with $\gamma e c + n = d\delta f$.
Write $\varrho^p(\alpha r)$ in the form $\alpha^p r^p$ with $\abs{\alpha^p}=k$, $r^p\in Q$.
By Lemma~\ref{L:tau-alpha-props}(a) there is $u^p\in Q_{e+p}$ with $\tau_{\alpha^p r^p}(u^p)=\bar u$.
Thus $\varrho(\alpha^p r^p u^p) = \bar \alpha \bar u$.
Further, repeated application of Lemma~\ref{L:tau-alpha-props}(b) gives the existence
of $u\in Q_c$ with $\varrho^p(\alpha r u)=\alpha^p r^p u^p$. Hence
$\varrho^{n}(\alpha r u) = \varrho^{\bar n}(\bar \alpha \bar u) = s$. So the lemma is true for any
$\alpha'=\alpha r\in Q^{k+1}$. By induction, the proof is finished.
\end{proof}

\begin{lemma}\label{L:F-trans-Ralpha-is-Q+2}
 For every $\alpha\in Q^+$ we have $R_\alpha=Q^+$.
\end{lemma}
\begin{proof}
Take any $\alpha\in Q_\gamma$ ($k=\abs{\alpha}\ge 1$) and $s\in Q$; let $d$ be such that $s\in Q_d$.
By Lemma~\ref{L:F-trans-Ralpha-is-Q+0}, for every $\delta\in\Sigma^k$
there are $n\ge 1$, $c\in\Sigma$ and $u\in Q_c$ with $\varrho^n(\alpha u)=s$ and $\gamma c+n=d\delta$.
Then, by Lemmas~\ref{L:F-trans-rho-itetate} and \ref{L:F-trans-rho-itetate1},
$$
 R_\alpha\supseteq
 \varrho^n\left( [\alpha u] \cap (Q_{\gamma c} \times Q^{l})   \right)
 \supseteq
 [s] \cap \left(  Q_{d\delta} \times Q^{l-k} \right)
$$
for every $l\ge k$.
Since $\delta\in\Sigma^k$ is arbitrary, we have
$$
 R_\alpha
 \supseteq
 [s] \cap \left(  Q_{d} \times Q^{l} \right)
 = [s]\cap Q^{l+1}
 \qquad\text{for every } l\ge 1.
$$
Thus $R_\alpha\supseteq [s]$ for every $s\in Q$, and so $R_\alpha=Q^+$.
\end{proof}

Lemmas~\ref{L:F-trans-Ralpha-is-Q+2} and \ref{L:F-trans-Ralpha-implies-trans} immediately give
the following proposition.

\begin{proposition}\label{P:F-trans}
The constructed map $F:X\to X$ is transitive.
\end{proposition}

\subsection{Periodic points of $F$}\label{SS:F-period}
\begin{lemma}\label{L:F-props-arc-in-Xm} Let $m\ge 1$ and
$x\in X^{(m)}$. Then for every sufficiently large $n$ there is $s_n\in Q$ such that
$F^n(x)\in A_\emptyword\cup A_{s_n}$;
moreover, $\lim_n s_n = 1$. Consequently, for every $x\in\bigcup_m X^{(m)}$,
$\omega_F(x)$ is either $\{a_\emptyword\}$ or  $\{b_\emptyword\}$.
\end{lemma}
\begin{proof}
First we show that for every $x\in X^{(m)}$ ($m\ge 2$)
there is $n\ge 0$ with $F^n(x)\in X^{(m-1)}$.
For if not, for every $n\ge 0$ there is $\alpha^n\in Q^m$ with $F^n(x)\in A_{\alpha^n}$
(use that $X^{(m)}$ is $F$-invariant).
Thus, by (\ref{EQ:rho-def}), $\alpha^{n}=\varrho^n(\alpha^0)$ for every $n$.
But then $\abs{\alpha^n}\to 1$ by Lemma~\ref{L:F-props-varrho}, a contradiction.

Applying this observation we obtain that for every $x\in X^{(m)}$ ($m\ge 2$) there is $n_0$
such that $F^{n_0}(x)\in X^{(1)}$. If there is $n_1$ with $F^{n_1}(x)\in X^{(0)}$ then
$F^n(x)\in X^{(0)}$ for every $n\ge n_1$ and we can take $(s_n)_n$ arbitrary such that $\lim s_n=1$.
Otherwise, for every $n\ge n_1$ there is $s_n$ with $F^n(x)\in A_{s_n}$. Since, by (\ref{EQ:def-F-X1}),
$s_{n+1}=\varrho(s_n)$ for every $n\ge 1$, $\lim s_n=1$ by (f) of Lemma~\ref{L:tau-0}.

The final assertion on $\omega$-limit sets of points $x$ from $X^{(m)}$ ($m\ge 1$)
now follows using the facts that $A_{s_n}\to\{b_\emptyword\}$ and
that $F|_{A_{\emptyword}}$ is given by $\varphi_\emptyword$,
an increasing homeomorphism satisfying $\varphi_\emptyword(t)<t$ for every $t\in (0,1)$.
\end{proof}

\begin{proposition}\label{P:F-period} For the map $F$ we have
$$
 \Per(F)=\Fix(F)=\{a_\emptyword,b_\emptyword\}.
$$
Thus $F$ has nowhere dense periodic points and has no periodic cut point.
\end{proposition}
\begin{proof}
Let $x$ be a periodic point of $F$. We claim that $x\not\in X^{(\infty)}$. For if not,
there is $\alpha\in Q^\infty$ with $x=b_\alpha\in \Per(F)$.
Then, by Lemma~\ref{L:def-RPD}(e),
for every $m$ there is $\gamma^m\in \Sigma^m$ such that $x$ belongs to the interior
of $D_{\gamma^m}$. Since every $\DDd_m$ is a regular periodic decomposition for $F$ (see Proposition~\ref{P:F-RPD}),
no such point $x$ is periodic for $F$, a contradiction.

Hence $x\in \bigcup_m X^{(m)}$. But then, by Lemma~\ref{L:F-props-arc-in-Xm},
$x\in\{a_\emptyword,b_\emptyword\}$.
Since the two points $a_\emptyword,b_\emptyword$ are fixed points of $F$, the assertion follows.
\end{proof}

\subsection{Almost openness of $F$}\label{SS:F-other}

\begin{lemma}\label{L:F-props-image-of-Xalpha-equalities}
Let $\alpha\in Q^+$, $\beta=\varrho(\alpha)$ and $\gamma$ be such that $\alpha\in Q_\gamma$. Then
$$
 F(X_\alpha)=
 \begin{cases}
   X_\beta
   &\text{if }\abs\alpha\ge 2 \text{ and } \alpha_1>\omega;
 \\
   A_\beta \cup \left(\bigsqcup_{s\in Q_{\last\gamma}} X_{\beta s}\right)
   &\text{if }\abs\alpha\ge 2 \text{ and } \alpha_1<\omega;
 \\
   [F(a_{\alpha}), a_\beta] \cup \left(\bigsqcup_{s\in Q_{\last\gamma}\cap(\varphi_\emptyword(\alpha),\beta]} X_{s}\right)
   &\text{if }\abs\alpha=1.
 \end{cases}
$$
\end{lemma}
\begin{proof} 
One inclusion was proved in Lemma~\ref{L:F-props-image-of-Xalpha}.
To prove the reverse inclusion assume first that $\abs{\alpha}\ge 2$;
recall the definition (\ref{EQ:alpha-cylinder}) of the cylinder $[\alpha]$.
Then $\bigcup_{\alpha'\in[\alpha]} A_{\alpha'}$ is dense in $X_\alpha$ and
$\bigcup_{\beta'\in\varrho([\alpha])} A_{\beta'}$ is dense in $F(X_\alpha)$.
Now the result follows from Lemma~\ref{L:F-trans-rho}.

Assume now that $\abs{\alpha}=1$ and write $X_\alpha$ in the form
$
    X_\alpha = A_\alpha
       \cup \left(\bigcup_{r\in Q} X_{\alpha r} \right)
$.
By the previous case and Lemma~\ref{L:tau-alpha-props},
$$
 \bigcup_{r>\omega} F(X_{\alpha r})
 = \bigcup_{r>\omega} X_{\beta \tau_\alpha''(r)}
 =\bigcup_{s\in Q} X_{\beta s}
 \supseteq X_{\beta}\setminus A_\beta
$$
and
$$
 \bigcup_{r<\omega} F(X_{\alpha r})
 = \bigcup_{c\in\Sigma}\bigcup_{r<\omega, r\in Q_c} \left\{
   A_{\tau_\alpha'(r)}
   \cup \left(
     \bigsqcup_{t\in Q_{\last{\gamma c}}} X_{\tau_\alpha'(r) t}
   \right)
 \right\}
 =\bigcup_{s\in Q_{\last{\gamma}}\cap(\varphi_\emptyword(\alpha),\beta)} X_{s}.
$$
Since $F(A_\alpha)=[F(a_\alpha),a_\beta]\cup A_\beta$ and $\beta\in Q_{\gamma+1}=Q_{\last{\gamma}}$,
the proof is finished.
\end{proof}

Lemma~\ref{L:F-props-image-of-Xalpha-equalities} immediately implies that $F$ is almost open.
Moreover, it also implies that $F$ is not open; for example,
the image $F(X_\alpha\setminus\{a_\alpha\})$ is not open for every $\alpha$ satisfying either
$\abs\alpha=1$ or $\abs{\alpha}\ge 2$, $\alpha_1<\omega$.
So we have the final part of the main theorem.

\begin{proposition}\label{P:F-almost-open}
The constructed map $F$ is almost open, but not open.
\end{proposition}

\subsection{Entropy of $F$}\label{SS:F-entropy}
\begin{proposition}\label{P:F-entropy}
The constructed map $F$ has positive entropy.
\end{proposition}
\begin{proof}
Fix any $m\in\ZZZ$ and let $c\in\Sigma$ be such that $z_m\in Q_c$.
By Lemma~\ref{L:tau-0}(d),
$\varphi_\emptyword(z_m)<z_{m-1}<z_{m+1}=\tau_\emptyword(z_m)$.
Thus
\begin{equation*}
   F(X_{z_m})
   \quad\supseteq\quad
   \bigcup_{s\in Q_{c+1}\cap(\varphi_\emptyword(z_m),\tau_\emptyword(z_m)]}  X_s
   \quad\supseteq\quad
   \left(X_{z_{m-1}}\cup X_{z_{m+1}}\right)
\end{equation*}
by Lemma~\ref{L:F-props-image-of-Xalpha-equalities}.
Since this is true for every $m$, we easily get
\begin{equation*}
    \left(F^2(X_{z_m}) \cap F^2(X_{z_{m+2}}) \right)
    \supseteq
    \left( X_{z_{m}}\cup X_{z_{m+2}} \right).
\end{equation*}
Hence, the (disjoint) sets $X_{z_m}, X_{z_{m+2}}$ form a horseshoe for $F^2$
and $h(F)\ge (1/2)\log 2$.
\end{proof}

\subsection{Invariant measures of $F$}\label{SS:F-invmeas}

\begin{proposition}\label{P:F-invmeas} 
Let $\mu$ be an $F$-invariant Borel probability measure. Then $\mu(X^{(m)}\setminus\{a_\emptyword,b_\emptyword\})=0$ for every $m$.
Moreover, if $\mu(\{a_\emptyword,b_\emptyword\})=0$, then
$\mu$ is non-atomic and
the dynamical system $(X,\BBb_X,\mu,F)$ is a measure-theoretic extension of the dyadic adding machine.
\end{proposition}
\begin{proof}
Fix any $\alpha\in Q^+$.
Since $\varrho^{-1}(Q^+)\subseteq Q^+$, $\varrho$ has no periodic cycles in $Q^+$
(Lemma~\ref{L:F-props-varrho}) and
$F^{-k}(A_\alpha)\subseteq \bigcup_{\beta\in\varrho^{-k}(\alpha)} A_\beta$ for every $k\ge 0$ by the construction of $F$,
the preimages $F^{-k}(A_\alpha)$ are pairwise disjoint. Thus $\mu(A_\alpha)=0$ by invariance and finiteness of $\mu$
and so $\mu(X^{(m)})=\mu(A_{\emptyword})$ for every $m$.

Assume that $c=\mu(A_{\emptyword})>0$. Fix $0<t<u<1$ and put $t'=\varphi_\emptyword^{-1}(t)$, $u'=\varphi_\emptyword^{-1}(u)$.
Since $(a_{t'},a_{s'})\subseteq F^{-1}((a_t,a_u)) \subseteq (a_{t'},a_{s'}) \cup (X^{(1)}\setminus A_\emptyword)$
and $\mu(X^{(1)}\setminus A_\emptyword)=0$,
the measure $\nu$ on $[0,1]$, defined by $\nu(B)=(1/c)\, \mu(\{a_t: t\in B\})$ for Borel $B\subseteq [0,1]$,
is $\varphi_\emptyword$-invariant. Since $\varphi_\emptyword(t)<t$ on $(0,1)$, the support of $\nu$
is contained in $\Omega(\varphi_\emptyword)=\{a_\emptyword,b_\emptyword\}$. Thus $\mu(A_\emptyword)=\mu(\{a_\emptyword,b_\emptyword\})$
and the first assertion is proved.

Assume now that $\mu(\{a_\emptyword,b_\emptyword\})=0$. Then $\mu(X^{(\infty)})=1$; note that
$X^{(\infty)}$ is $F$-invariant.
Define $\pi:X^{(\infty)}\to\Sigma^\infty$ by $b_\alpha\mapsto \gamma$, where $\gamma\in\Sigma^\infty$ is such that
$\alpha\in Q_\gamma$. Then $\pi$ is a Borel measurable surjection
(in fact, for every $\gamma\in\Sigma^*$ the preimage of the $\gamma$-cylinder
is $D_\gamma\cap X^{(\infty)}$).
Let $G$ denote the dyadic adding machine on $\Sigma^\infty$
and $\nu$ denote the unique $G$-invariant probability on $\Sigma^\infty$.
From Proposition~\ref{P:F-RPD} and the fact that $\mu(X^{(\infty)})=1$
it follows that $\mu$ is non-atomic, $\pi\circ F=G\circ\pi$ and $\nu=\pi_{\#}\mu$.
Thus $\pi:(X,\BBb_X,\mu,F)\to (\Sigma^\infty,\BBb_{\Sigma^\infty},\nu,G)$ is a factor map.
\end{proof}

\section{Proof of the main theorem}\label{S:proof}
In this section we prove that factor maps $G_m$ ($m\ge 2$) of $F$ obtained
by collapsing $X^{(m)}$ to a point satisfy all the properties stated in Theorem~\ref{T:main}.
We first construct a special convex metric on $X$, see Section~\ref{SS:G-convex-metric}.
Then, in Section~\ref{SS:G-lipschitz} we show that the factor maps $G_m$
are Lipschitz and $\lim_m \Lip(G_m)=1$. Finally, in Section~\ref{SS:G-proof}
we give a proof of Theorem~\ref{T:main}.

\subsection{A convex metric on $X$}\label{SS:G-convex-metric}

\begin{lemma}\label{L:rho-finite-to-one}
The map $\varrho|_{Q^*}:Q^*\to Q^*$ is finite-to-one; that is,
for every $\beta\in Q^*$ there are at most finitely many $\alpha\in Q^*$ with $\varrho(\alpha)=\beta$.
\end{lemma}
\begin{proof}
If $\beta=\emptyword$ then $\varrho^{-1}(\beta)=\{\emptyword\}$, and
if $\beta\in Q$ then $\varrho^{-1}(\beta)=\tau_\emptyword^{-1}(\beta)$ is either singleton or a two-point set, see
Lemma~\ref{L:tau-0}(a). Assume now that $m\ge 2$ and that $\varrho^{-1}(\beta)$ is finite whenever $\abs{\beta}<m$.
Take any $\beta=\beta' s\in Q^m$ and any $\alpha=\alpha' r\in\varrho^{-1}(\beta)$, where $r,s\in Q$.
Then $\varrho(\alpha')=\beta'$ by (\ref{EQ:rho-props}) and so,
by the hypothesis, $\alpha'$ belongs to the finite set $\varrho^{-1}(\beta')$.
Further, $s=\tau_{\alpha'}(r)$ again by  (\ref{EQ:rho-props});
since every point from $Q$ has at most finitely many $\tau_{\alpha'}$-preimages by Lemma~\ref{L:tau-12}(d),
for any given $\alpha'$ the index $r$ must belong to a finite set $\tau_{\alpha'}^{-1}(s)$.
Thus $\varrho^{-1}(\beta)$ is finite and the proof is finished.
\end{proof}

\begin{lemma}\label{L:F-entropy-metric}
There is a convex metric $d$ on $X$ which is compatible with the topology of $X$ and is such that
\begin{equation}\label{EQ:F-entropy-metric-Lip-Falpha}
    \Lip(F|_{X_\alpha})\le L_m^2
    \qquad\text{for every } \alpha\in Q^*, \abs{\alpha}\ge 2.
\end{equation}
\end{lemma}
\begin{proof}
For convenience put $a_\alpha=b_\alpha$ for $\alpha\in Q^\infty$
and $a_\alpha = a_{\bar\alpha}$
for $\alpha=\bar{\alpha} 0^\infty$ with $\bar\alpha\in Q^*\times[0,1]$.
Thus, for every $x\in X$ we have unique $\alpha\in \tilde{Q}=(Q^*\times[0,1]\times\{0\}^\infty)\cup Q^\infty$
such that $x=a_\alpha$.

First we inductively define `lengths' $\lambda_\alpha$ of the arcs $A_\alpha$ ($\alpha\in Q^*$),
using which the convex metric $d$ will be defined.
Put $\lambda_\emptyword=1$ and for $r\in Q$ define $\lambda_r$ arbitrarily in such a way
that $0<\lambda_r<1$ for every $r$ and $\{\lambda_r:r\in Q\}$ is null
(that is, for every $\eps>0$
there are only finitely many $r$'s with $\lambda_r>\eps$).
Assume that $m\ge 2$ and that $\lambda_\alpha$ has been defined for every $\abs\alpha<m$; take any
$\alpha\in Q^m$. Let $p_\alpha\ge 1$ denote the minimal integer $p$ such that $\abs{\varrho^p(\alpha)}<m$.
We define
\begin{equation}\label{EQ:F-entropy-metric-lengths-def}
    \lambda_\alpha = \frac{\lambda_\beta}{L_m^{p_\alpha}}
    \qquad
    \text{where } \beta = \varrho^{p_\alpha}(\alpha);
\end{equation}
note that, since $\abs\beta<m$, $\lambda_\beta$ has already been defined.
After finishing the induction we obtain $\lambda_\alpha$ for every $\alpha\in Q^*$ in such a way that
\begin{equation}\label{EQ:F-entropy-metric-lengths}
    0<\lambda_\alpha\le \prod\limits_{2\le k\le m} L_m^{-1}
    \qquad\text{and}\qquad
    \frac{\lambda_\alpha}{\lambda_{\varrho(\alpha)}} = L_m^{-1}
    \qquad\text{for every }   \alpha\in Q^m, m\ge 2.
\end{equation}
Moreover, we claim that
\begin{equation}\label{EQ-lambdas-are-null}
\{\lambda_\alpha:\alpha\in Q^*\} \quad\text{is null}.
\end{equation}
To prove this it suffices to show that, for any fixed $\eps>0$,
the union of the sets $R_m=\{\alpha\in Q^m: \lambda_\alpha>\eps\}$
($m\ge 1$) is finite.
By (\ref{EQ:F-entropy-metric-lengths}) and (\ref{EQ:Lm-def}) we have that there is $m_0\ge 2$ such that
$R_m=\emptyset$ for every $m>m_0$. Thus we need only to prove that every $R_m$ is finite.
For $m=1$ we have this immediately by the choice of $\lambda_r$ for $r\in Q$.
Assume now that $m>1$ and that $R_{m-1}$ is finite.
Since $L_m>1$, (\ref{EQ:F-entropy-metric-lengths-def}) implies the existence of
$p_m$ such that $p_\alpha\le p_m$ for every $\alpha\in R_m$. Further,
$\varrho^{p_\alpha}(\alpha)\in R_{m-1}$ for every such $\alpha$. By finiteness of $R_{m-1}$ and the fact
that $\varrho$ (and hence all its iterates) is finite-to-one by Lemma~\ref{L:rho-finite-to-one},
we obtain that $R_m$ is finite. Thus, by induction, the system $\{\lambda_\alpha:\alpha\in Q^*\}$ is null.

Now we define the metric $d$. Basically, this is the unique convex metric on $X$ such that
$d(a_{\alpha s},a_{\alpha t})=\abs{t-s} \lambda_\alpha$ for any $\alpha\in Q^*$ and $s,t\in [0,1]$.
But since dendrites are uniquely arcwise connected, we can also easily give an explicit
formula for $d(x,y)$ for any different $x=a_\alpha,y=a_\beta\in X$ ($\alpha,\beta\in \tilde Q$):
\begin{equation}\label{EQ:F-entropy-metric-def}
  d(x,y)
  =\abs{\alpha_k-\beta_k}\lambda_{\alpha_0\dots\alpha_{k-1}}
  +
  \sum_{i> k} \left(
    \alpha_i\lambda_{\alpha_0\dots\alpha_{i-1}}
    + \beta_i\lambda_{\beta_0\dots\beta_{i-1}}
  \right)
\end{equation}
where $k=k_{\alpha,\beta}\ge 0$ is the first index $i$ with $\alpha_i\ne\beta_i$
(note that $\alpha_0\dots\alpha_{k-1}=\emptyword$ for $k=0$).

Trivially $d$ is a well-defined convex metric on $X$. The compatibility of $d$ with the topology of $X$
can be easily seen as follows.
Take points $x=a_\alpha, x_n=a_{\alpha^n}$ ($\alpha,\alpha^n\in \tilde{Q}$, $\alpha^n \ne\alpha$) for $n\ge 0$ in $X$
and put $k_n=k_{\alpha,\alpha^n}$ for every $n$; assume that $(k_n)_n$ converges.
Then $x_n\to x$ in $X$ if and only if either $k_n\to\infty$, or
$k_n$ is eventually equal to a constant $k\ge 0$, $\alpha_i=0$ for every $i>k$ and
$\lim_n \alpha^n_k = \alpha_k$.
(The first case occurs when $x\in X^{(\infty)}$, the second one when $x\in\bigcup_m X^{(m)}$.)
By (\ref{EQ-lambdas-are-null}), (\ref{EQ:F-entropy-metric-lengths}) and (\ref{EQ:Lm-def})
this is equivalent to $\lim_n d(x_n,x)=0$.
(In fact, if $k_n\to\infty$ then
(\ref{EQ:F-entropy-metric-lengths}) and (\ref{EQ:Lm-def}) give $d_n=d(x_n,x)\to 0$.
Assume now that $k_n=k$ for every sufficiently large $n$. If $\alpha_i>0$ for some $i>k$,
or $\alpha_k^n\not\to\alpha_k$, then $d_n\not\to 0$. Otherwise, for any $\eps>0$,
(\ref{EQ:F-entropy-metric-lengths}) and (\ref{EQ:Lm-def}) give the existence
of $m>k$ such that $\sum_{i>m} \lambda_{\alpha^n_0\dots\alpha^n_{i-1}}<\eps$ for every $n$;
further,
$\lambda_{\alpha^n_0\dots\alpha^n_{i-1}}\to 0$ for every $k<i\le m$
by (\ref{EQ-lambdas-are-null}) together with $\alpha_k^n\to\alpha_k$ and $\alpha_k^n\ne\alpha_k$.
Thus $d_n\to 0$.)
Hence $d$ is compatible with the topology of $X$.

\smallskip

Finally we show (\ref{EQ:F-entropy-metric-Lip-Falpha}). First we prove that
\begin{equation}\label{EQ:F-entropy-metric-Lip-FAalpha}
    \Lip(F|_{A_\alpha})\le L_m^2
    \qquad\text{for every } \alpha\in Q^*, \abs{\alpha}\ge 2.
\end{equation}
To this end, fix any $\alpha\in Q^m$ ($m\ge 2$),
put $\beta=\varrho(\alpha)$
and recall that $\tau_\alpha:[0,1]\to [0,1]$ is Lipschitz-$L_m$.
Then, for any $t,s\in [0,1]$, $d(a_{\alpha t}, a_{\alpha s}) = \lambda_\alpha\abs{t-s}$
by (\ref{EQ:F-entropy-metric-def}),
and
\begin{equation*}
    d(F(a_{\alpha t}), F(a_{\alpha s}))
    =d(a_{\beta \tau_\alpha(t)}, a_{\beta \tau_\alpha(s)})
    =\lambda_\beta \abs{\tau_\alpha(t)-\tau_\alpha(s)}
    \le \lambda_\beta L_m \abs{t-s}
\end{equation*}
by (\ref{EQ:def-F-X2}) and (\ref{EQ:F-entropy-metric-def}).
So $\Lip(F|_{A_\alpha})\le L_m \lambda_\beta/\lambda_\alpha= L_m^2$
by (\ref{EQ:F-entropy-metric-lengths}).

Now (\ref{EQ:F-entropy-metric-Lip-Falpha}) follows from (\ref{EQ:F-entropy-metric-Lip-FAalpha})
from the convexity of $d$ and the fact that the sequence $(L_m)_m$ is decreasing.
In fact, take any $\alpha\in Q^*$ with $\abs{\alpha}=m\ge 2$ and any different $x,y\in X_\alpha$;
we want to show that $d(F(x),F(y))\le L_m^2 d(x,y)$.
Since $\bigcup_p X^{(p)}$ is dense in $X$, we may assume that $x,y\in X^{(p)}$ for some $p$.
And since the set $Y=X^{(p)}\cap X_\alpha$ is a dendrite, the arc $[x,y]$ is contained in $Y$
and we can find $x_0,\dots,x_l\in Y$, $\beta^0,\dots,\beta^{l-1}\in Q^*$ such that $x_0=x$, $x_l=y$,
$[x_{i},x_{i+1}]\subseteq A_{\alpha\beta^i}$ and
the intersections $[x_{i},x_{i+1}]\cap [x_{j},x_{j+1}]$
are empty for $j>i+1$ and equal to $\{x_{i+1}\}$ for $j=i+1$.
Then, by convexity of $d$ and unique arcwise connectedness of $X$,
$d(x,y) = \sum_{i<l} d(x_i, x_{i+1})$. Application of (\ref{EQ:F-entropy-metric-Lip-FAalpha})
gives
$$
 d(F(x),F(y))\le \sum_{i<p} d(F(x_i),F(x_{i+1}))
 \le \sum_{i<p} L_{m+\abs{\beta^i}}^2 d(x_i,x_{i+1})\le L_m^2 d(x,y).
$$
Thus, (\ref{EQ:F-entropy-metric-Lip-Falpha})
is proved.
\end{proof}

\subsection{Factor maps $G_m$}\label{SS:G-lipschitz}

Fix $m\ge 0$.
Let $Y_m$ be the factor space obtained from $X$ by collapsing $X^{(m)}$ to a point and let $\pi_m:X\to Y_m$
be the natural projection. (Note that $Y_m$ is a dendrite which can be described as an $\omega$-star of
the universal dendrites of order $3$.) Equip $Y_m$ with the convex metric $d_m$ given by
\begin{equation}\label{EQ:dm-def}
 \begin{split}
    d_m(\pi_m(x),\pi_m(x'))
    &= d(x,x')-d(\rho_m(x),\rho_m(x'))
 \\
    &=
    \begin{cases}
        d(x,\rho_m(x)) + d(x',\rho_m(x'))  &\text{if } \rho_m(x)\ne \rho_m(x'),
    \\
        d(x,x')                     &\text{if } \rho_m(x)= \rho_m(x'),
    \end{cases}
 \end{split}
\end{equation}
where $\rho_m:X\to X^{(m)}$ is the first point retraction.
We omit the easy proof of the fact that $d_m$ is a well-defined convex metric compatible with the topology of $Y_m$.
Since $X^{(m)}$ is invariant for $F$, we may define the factor map
\begin{equation}\label{EQ:Gm-def}
    G_m:Y_m\to Y_m,\qquad
    G_m(\pi_m(x)) = \pi_m(F(x)).
\end{equation}

\begin{lemma}\label{L:F-entropy-factor-lipschitz}
Let $m\ge 2$ and let $G_m:Y_m\to Y_m$ be defined by (\ref{EQ:Gm-def}).
Then $\Lip(G_m)\le L_m^2$.
\end{lemma}
\begin{proof}
Put $G=G_m$.
Take any $y,y'\in Y_m$. 
If there is $\alpha\in Q^m$ such that $y,y'\in \pi_m(X_\alpha)$ then
we can find $x\in \pi_m^{-1}(y)$, $x'\in \pi_m^{-1}(y')$ such that $x,x'\in X_\alpha$; hence, since
$\rho_m(x)=\rho_m(x')$,
\begin{equation}\label{EQ:F-entropy-factor-lipschitz}
d_m(G(y),G(y')) \le  d(F(x),F(x'))\le L_m^2 d(x,x')=L_m^2 d_m(y,y').
\end{equation}
Otherwise put $y''=\pi_m(X^{(m)})$; since $y''\in \pi_m(X_\alpha)$ for every $\alpha\in Q^m$,
(\ref{EQ:F-entropy-factor-lipschitz}) and convexity of $d_m$ give
$d_m(G(y),G(y')) \le d_m(G(y),G(y'')) + d_m(G(y''),G(y')) \le L_m^2(d_m(y,y'')+d_m(y'',y'))
=L_m^2 d_m(y,y')$.
\end{proof}

\begin{lemma}\label{L:F-entropy-factor-equal-entropy}
Let $m\ge 0$ and let $G_m:Y_m\to Y_m$ be defined by (\ref{EQ:Gm-def}).
Then $h(G_m)=h(F)$.
\end{lemma}
\begin{proof}
Since $G_m$ is a factor of $F$ we have $h(G_m)\le h(F)$.
By Bowen's formula \cite[Theorem~17]{Bowen}, $h(F)\le h(G_m)+\sup_{y\in Y_m} h(F,\pi_m^{-1}(y))$.
So also $h(F)\le h(G_m)$ since the fibre entropies $h(F,\pi_m^{-1}(y))$ are all zero;
in fact, either $\pi_m^{-1}(y)$ is a singleton, or $\pi_m^{-1}(y)=X^{(m)}$ and $h(F,X^{(m)})=h(F|_{X^{(m)}})=0$
by Lemma~\ref{L:F-props-arc-in-Xm}.
\end{proof}

\subsection{Proof of the main theorem}\label{SS:G-proof}
Now we are ready to prove Theorem~\ref{T:main}. Let $L>1$ be arbitrary. Take $m\ge 2$ such that
$L_m^2\le L$ and put $(X,f)=(Y_m,G_m)$. Then $f$ satisfies the property (g) from the theorem
by Lemma~\ref{L:F-entropy-factor-equal-entropy},
and also the property (f) with $d=d_m$ by Lemma~\ref{L:F-entropy-factor-lipschitz}.
Since $f$ is a factor of $F$, Proposition~\ref{P:F-trans} gives that $f$ is transitive.
The fact that no transitive point of $f$ is a cut point of $X$ follows
from Lemma~\ref{L:F-props-arc-in-Xm};
hence we have (a).
From Proposition~\ref{P:F-RPD} easily follows that, for every $k\ge 0$,
$(\pi_m(D_\gamma))_{\gamma\in\Sigma^k}$ is a regular periodic decomposition for $f$;
thus we have (b). Proposition~\ref{P:F-period} implies that $f$ has a unique periodic point
and we have (c). Since the `collapsed' subdendrite $X^{(m)}$ is nowhere dense,
almost openness  of $F$ (Proposition~\ref{P:F-almost-open}) immediately gives that $f$ is
almost open; $f$ is not open by the same reasoning as that used for $F$, so we have (d).
Finally, (e) follows from Proposition~\ref{P:F-invmeas} and
the fact that, for any non-atomic invariant measure $\mu$ of $F$,
the measure-theoretic factor induced by $\pi_m$ is isomorphic
to $(X,\BBb_X,\mu,F)$.
Thus Theorem~\ref{T:main} is proved.

\subsection*{Acknowledgements}
The author is greatly indebted to Tomasz Downarowicz, Sergiy Ko\-lya\-da, Dominik Kwietniak and \mL{}ubom\'ir Snoha
for several helpful comments concerning the subject of the paper.
Supported by the Slovak Research and Development Agency under contract no.~APVV-0134-10
and by the Slovak Grant Agency under grant no.~VEGA~1/0978/11.

\newcommand{\arxiv}[2]{\href{http://arxiv.org/abs/#1}{arXiv:#1 [#2]}}

\end{document}